\documentclass[10pt]{article}

\usepackage{array}
\usepackage{authblk}
\usepackage{amsmath}
\usepackage{amsthm}
\usepackage{amssymb}
\usepackage{bm}
\usepackage{diagbox}
\usepackage{dirtytalk}
\usepackage{longtable}
\usepackage{mathtools}
\usepackage{multirow}
\usepackage{pifont}
\usepackage{tabularx}
\usepackage{thmtools}
\usepackage{thm-restate}
\usepackage{tikz-cd}

\usepackage[style=numeric,sorting=none]{biblatex}
\usepackage{hyperref}
\usepackage{url}
\hypersetup{
    colorlinks=true,
    urlcolor=blue,
    citecolor=blue,
    linkcolor=red
}

\declaretheoremstyle[
    headfont=\bfseries, 
    notefont=\bfseries, 
    notebraces={(}{)}, 
    bodyfont=\normalfont, 
]{boldstyle}

\declaretheorem[name=Theorem,style=boldstyle,numberwithin=subsection]{theorem}
\declaretheorem[name=Definition,style=boldstyle,numberlike=theorem]{definition}
\declaretheorem[name=Proposition,style=boldstyle,numberlike=theorem]{proposition}
\declaretheorem[name=Lemma,style=boldstyle,numberlike=theorem]{lemma}

\declaretheorem[name=Example,style=boldstyle,numberlike=theorem]{example}
\declaretheorem[name=Note,style=boldstyle,numberlike=theorem]{note}

\begin{document}

\newcommand{\subsectiontexta}{Encodings in Arithmetic Logic and Social Choice Theory}
\newcommand{\subsectiontextb}{Self-Reference Systems}
\newcommand{\subsectiontextc}{The Fixed-Point Property}
\newcommand{\subsectiontextd}{Embeddable Self-Reference Systems}
\newcommand{\subsectiontexte}{The Abstract Diagonalisation Lemma}
\newcommand{\subsectiontextf}{Consistency and Incomputability}
\newcommand{\subsectiontextg}{An Abstract Incomputability Argument}

\newcommand\ldoteq{\mathrel{\ensurestackMath{%
            \stackengine{-.5ex}{\lessdot}{-}{U}{c}{F}{F}{S}}}}

\newcommand\sbullet[1][.75]{\mathbin{\vcenter{\hbox{\scalebox{#1}{$\bullet$}}}}}

\newcommand{\N}{\mathbb{N}}

\newcommand{\Obar}[1]{\ulcorner #1 \urcorner}
\newcommand{\Obarvar}[0]{\Obar{-}}

\newcommand{\NGodel}[1]{G'(#1)}
\newcommand{\NGodelvar}[0]{G'}

\newcommand{\enc}[0]{\text{enc}}
\newcommand{\emb}[0]{\text{emb}}
\newcommand{\diag}[0]{\text{diag}}
\newcommand{\app}[0]{\text{app}}
\newcommand{\comp}[0]{\sbullet}

\newcommand{\const}[0]{\Join}
\newcommand{\inconst}[0]{\not\const}
\newcommand{\Provable}[0]{Provable}

\newcommand{\DObar}[1]{\ulcorner\mkern-7mu\ulcorner\mkern-3mu#1%
    \mkern-3mu\urcorner\mkern-7mu\urcorner}
\newcommand{\DObarvar}[0]{\DObar{-}}

\newcommand{\A}{\mathcal{A}}
\newcommand{\T}{\Upsilon}
\newcommand{\Theory}{\mathcal{T}}

\newcommand{\F}{\mathcal{F}}
\newcommand{\C}{\mathcal{C}}
\newcommand{\D}{\mathcal{D}}
\newcommand{\E}{\mathcal{E}}
\newcommand{\G}{\mathcal{G}}
\renewcommand{\L}{\mathcal{L}}

\renewcommand{\c}{\textbf{c}}
\renewcommand{\i}{\textbf{i}}

\newcommand{\uc}{\underline{\textbf{c}}}
\newcommand{\ui}{\underline{\textbf{i}}}

\renewcommand{\P}{\mathcal{P}}
\newcommand{\PN}{\P^N}

\newcommand{\PC}{\underline{\P}}
\newcommand{\PCN}{\underline{\P}^N}
\newcommand{\Welfare}{w}

\newcommand{\join}{+}
\newcommand{\bwedge}{\bm{\wedge}}

\newcommand{\indiff}{e}

\newcommand{\term}[1]{\textit{\textbf{#1}}}
\newcommand{\saymath}[1]{\text{\say{$#1$}}}

\newcommand{\tick}{\ding{51}}
\newcommand{\cross}{\ding{55}}
\newcommand{\sometimes}{Satisfiable}


\newcommand{\parsplit}{\par\vspace{6pt}}
\newcommand{\finestsection}[1]{\noindent\textbf{#1}\parsplit}

\newenvironment{oneshot}[1]{\@begintheorem{#1}{\unskip}}{\@endtheorem}

\title{Comparing and Contrasting Arrow's Impossibility Theorem and Gödel's Incompleteness Theorem}
\date{November 10, 2025}

\renewcommand\Affilfont{\small}

\author[1,2]{Ori Livson \thanks{Corresponding author: ori.livson@sydney.edu.au}}
\author[1,2]{Mikhail Prokopenko}

\affil[1]{The Centre for Complex Systems,  University of Sydney, NSW 2006, Australia}
\affil[2]{School of Computer Science, Faculty of Engineering, University of Sydney, NSW 2006, Australia}

\maketitle

\vspace{-24pt}

\begin{abstract}
    Incomputability results in formal logic and the Theory of Computation (i.e., incompleteness and undecidability) have deep implications for the foundations of mathematics and computer science. Likewise, Social Choice Theory, a branch of Welfare Economics, contains various impossibility results that place limits on the potential fairness, rationality and consistency of social decision-making processes. However, a relationship between the fields' most seminal results: Gödel's First Incompleteness Theorem of formal logic, and Arrow's Impossibility Theorem in Social Choice Theory is lacking. In this paper, we address this gap by introducing a general mathematical object called a \textit{Self-Reference System}. Correspondences between the two theorems are formalised by abstracting well-known diagonalisation and fixed-point arguments, and consistency and completeness properties of provability predicates in the language of Self-Reference Systems. Nevertheless, we show that the mechanisms generating Arrovian impossibility and Gödelian incompleteness have subtle differences.
\end{abstract}

\section{Introduction}

\subsection{Incomputability}

Incomputability refers to the concept in computer science and mathematics in which a problem is fundamentally unsolvable. Examples include the existence of true but unprovable statements (e.g., Gödel's (First) Incompleteness Theorem~\cite{godel}), problems no algorithm can solve for all inputs (e.g., the undecidability of the Halting Problem~\cite{turing}), or problems where attempting to evaluate a property leads to a contradiction (e.g., Russell's Paradox~\cite{principles-of-mathematics}). Incomputability is used as an umbrella term in logic and computer science~\cite{hoare-incomputability} as well as in the social sciences to describe phenomena deemed to be unpredictable or incalculable~\cite[p. vii]{incomputability-political-economy}. In this study, incomputability is understood as the inability to establish a formal property in a system without encountering a logical contradiction, i.e., without necessarily employing algorithmic notions of computability. The use of the term (along with \say{Uncomputable}) has broadened to include physics and biology~\cite{tangled-hierarchies,the-incomputable-reality,longo-incomputability,incomputability-experiment}, and Complex Systems theory~\cite{prokopenko-liar-paradox,casti91,casti94}.

\parsplit

Incomputability results have profound implications for the foundations of mathematics, highlighting the limits of formal systems~\cite{hoare-incomputability}. Incomputability is often established using \textit{Diagonalisation} and \textit{Fixed-Point} arguments \cite{jacquette-diagonalisation,simmons-good-bad-diagonalisation,yanofsky}. These arguments have been generalised to demonstrate that various unsolvable problems are examples of abstract Diagonalisation and Fixed-Point arguments~\cite{gaifman,prokopenko-liar-paradox,smullyan-diagonalization,smullyan-theory-formal-systems}. However, Diagonalisation and Fixed-Point arguments are rarely applied to the problems outside of Computer Science and Logic such as Social Decision-Making and Complex Systems Theory. Some notable exceptions include the study of Universal Spin Systems~\cite{gonda} and the Brandenburger-Keisler paradox of Epistemic Game Theory~\cite{brandenburger}.

\parsplit

Moreover, many incomputability results in Social Decision-Making and Complex Systems Theory are  called \say{impossibility} or \say{no-go} results (e.g.,\\\cite{arrow-abramsky, wolpert-monotheism}) rather than incomputability results. A notable impossibility result in Social Choice Theory is Arrow's Impossibility Theorem~\cite{arrow-original}, which demonstrates the inability to devise a ranked-choice voting method that jointly satisfies a number of intuitive and seemingly desirable constraints. In this paper, we investigate the relationship between Arrow's Impossibility Theorem and Gödel's Incompleteness Theorem, by comparing and contrasting the mechanisms that generate impossibility in Social Choice Theory and incompleteness in Logic. This in done by introducing a general mathematical object called a Self-Reference System.

\subsection{Arrow's Impossibility Theorem}\label{introduction:arrow}

Arrow's Impossibility Theorem is a seminal result in Social Choice Theory, a branch of Economics that studies methods of aggregating individual inputs (e.g., votes, judgements, utility, etc.) into group outputs (e.g., election outcomes, sentencings, policies)~\cite{sep-social-choice}. Social Choice Theory is valuable in its ability to study how social-decision making \textit{can} be done, rather than how \textit{it is} done~\cite{possibility-of-social-choice}. Arrow's Impossibility Theorem reveals limitations in ranked-choice decision-making, which challenge economists' and policymakers' assumptions about the possibility of perfectly fair, rational, and consistent method for collective decision-making processes.

\parsplit

In short, Arrow's Impossibility Theorem  states that any ranked-choice decision-making process that satisfies two specific constraints (Unanimity and Independence of Irrelevant Alternatives) either fails to always produce a valid outcome or has a \term{dictator}. A dictator is a fixed individual that no aggregate outcome ever contradicts. In other words, the group must always share the dictator's preferences, whenever they are strict (i.e., the dictator is not indifferent to them). The existence of a dictator is a significant limitation on the outcomes attainable by ranked-choice decision-making processes.

\parsplit

Moreover, by generalising D'Antoni's work~\cite{dantoni}, Livson and Prokopenko \cite{paper0-arxiv} have established that Non-Dictatorship given the other constraints specifically entails that aggregation leads to a contradictory preference cycle for some collection of individual preferences, i.e., an outcome where for alternatives $X$, $Y$ and $Z$: $X$ is strictly preferred to $Y$, $Y$ to $Z$, and $Z$ to $X$. In this paper, we will establish formal relationships between incompleteness in logical theories of arithmetic and impossibility with respect to aggregation that produces contradictory preference cycles. Moreover, we establish a relationship between individual preferences that aggregate to contradictory preference cycles and Gödel sentences in logic, which in essence assert their own non-provability. This is significant due to the pervasiveness of circular and self-referential paradoxes in occurrences of incomputability~\cite{yanofsky,tangled-hierarchies, prokopenko-liar-paradox}.

\parsplit

Incomputability results in Social Choice Theory have been established in other contexts, related to Arrow's Impossibility Theorem. For example, although Fishburn's Possibility Theorem~\cite{fishburn} demonstrates that a dictator is not necessitated in the Arrovian framework when allowing for infinitely many individuals, Mihara~\cite{mihara} shows a dictator is still necessitated when restricting to a certain notion of \textit{computable} decision-making processes~\cite{hall-mihara}. Other examples of incomputability results in Social Choice Theory include Parmann~\cite{parmann} proving that certain modal logics which model strategic voting are undecidable, and Tanaka~\cite{tanaka-arrow-diagonalisation} showing of the problem of determining whether a decision-making process has a dictator is undecidable. However, a formal relationship between the original (i.e., finite) Arrow's Impossibility Theorem and Gödel's Incompleteness Theorem has not been established to date. In this paper, we aim to address this gap by expressing impossibility in the sense of Arrow, and Incompleteness in the sense of Gödel in the same terms in a more general theory of incomputability.

\subsection{Summary of Results}

Our framework employs a generalised notion of \term{encoding} --- a function from a set of \term{expressions} to a set of \term{constants}. In logical theories of arithmetic such as Peano Arithmetic (henceforth called \term{Arithmetic Logic}), expressions are well-formed formulae, constants are numbers, and encoding is Gödel numbering, which assigns each formula a unique numeric code. In Social Choice Theory, constants are preference relations (e.g., of a single individual or an aggregate outcome), expressions are profiles --- a finite collection of individual preference relations, and encoding is a Social Welfare Function, which assigns each profile an aggregate preference relation. Importantly, an aggregate outcome understood as an encoding of individual preferences is highly \term{lossy}. This is in contrast to Gödel numbering, which is lossless because one can always decode a Gödel number back to its underlying formula.

\parsplit

We additionally employ a mechanism for applying expressions to encodings, called an \term{application function}, e.g., substituting a Gödel number for a free variable of a predicate, or comparing an individual's preferences to an aggregate's preferences. Diagonalisation arises out of application of an expression to its own encoding. Insofar as an encoding refers to (or is coupled with) its underlying expression, diagonalisation is self-referential. We call a choice of encoding and application function a \term{Self-Reference System}.

\parsplit

The correspondence between in Gödel and Arrow's theorems is established in terms of a shared role played by Gödel Sentences, and by profiles that aggregate to contradictory preference cycles in the Self-Reference Systems instantiated in each domain. This correspondence is formalised by abstracting to the language of Self-Reference Systems: diagonalisation and fixed-point arguments, and properties applicable to provability predicates such as completeness and the existence of Gödel sentences. Nevertheless, the extent of this correspondence reveals interesting differences between the mechanisms that underlie incomputability and impossibility in each domain.

\parsplit

We begin by showing that the well-known Diagonalisation Lemma used in proofs of Gödel's Incompleteness Theorem, and a Social Welfare Function having a dictator corresponds to the existence of a \term{diagonaliser} in a Self-Reference System (Theorems \ref{theorem:arithmetic-logic-fixed-points}-\ref{theorem:dictator-as-diagonaliser}). The fixed-point property that underlies the Diagonalisation Lemma is a key element of Gödel's Incompleteness Theorem due to its entailment of Gödel sentences. However, this property (i.e., dictatorship) holding plays the opposite role in Arrow's Impossibility Theorem because it prevents aggregation to contradictory preference cycles.

\parsplit

A more significant overlap arises through the abstraction of properties applicable to provability predicates. Specifically, we restate constraints on provability predicates (i.e., equations concerning Lindenbaum Algebras) using the general language of Self-Reference Systems. In logic, the existence of Gödel Sentences renders the consistency and completeness of a theory as being mutually exclusive. When a pair of expressions within a Self-Reference System resembles, in abstract terms, a Gödel sentence and a provability predicate --- yet the corresponding notions of consistency and completeness are mutually exclusive --- we refer to that Self-Reference System as a \term{quasi-Gödelian system}. Specifically, a quasi-Gödelian system is defined with respect to an expression $\T$ that abstracts the role of a provability predicate. A Self-Reference System is a quasi-Gödelian system with respect to $\T$ if there exists another expression $d$ such that the pair $(\T,d)$ jointly satisfies a number of abstract properties in a manner related to Gödel's Incompleteness Theorem. Namely, that $d$ is a \term{quasi-Gödel sentence} with respect to $\T$, and that for $(\T, d)$, two \term{quasi-consistency} and \term{quasi-completeness} constraints are mutually exclusive. We use the prefix~\say{quasi-} to emphasise that, although for example, quasi-completeness does not directly pertain to provability-based completeness, it encapsulates a significant algebraic property that underpins the notion of completeness.

\parsplit

Both Gödel's Incompleteness Theorem and Arrow's Impossibility Theorem are shown to correspond to statements about the existence of quasi-Gödelian systems. In logic, the existence of Gödel sentences yields that a corresponding Self-Reference System is a quasi-Gödelian system with respect to its provability predicate (Theorem \ref{theorem:godel-abstract}). In Social Choice Theory, the delineating of when a particular Self-Reference System is a quasi-Gödel system (Theorems \ref{theorem:condorcet-abstract}-\ref{theorem:dictator-abstract-2}) mirrors the characterisation of Arrow’s Impossibility Theorem wherein a Social Welfare Function either produces contradictory preference cycles or has a dictator~\cite{dantoni,paper0-arxiv}.

\parsplit

Thus demonstrating that key properties underlying incompleteness in Arithmetic Logic characterise the structure of certain Social Welfare Functions that produce contradictory preference cycles reveals a striking overlap between incomputability and impossibility between Gödel and Arrow's Theorems.

\section{Background}\label{section:background}

In Sections \ref{subsection:background-godel} and \ref{subsection:background-algebraic-logic}, we provide a background in Arithmetic Logic focusing on Gödel's Incompleteness Theorem and Algebraic Logic, respectively. Then, in Sections \ref{subsection:background-arrow} and Section \ref{subsection:condorcet-paradox} we provide a background in Social Choice Theory focusing on Arrow's Impossibility Theorem and Condorcet's Paradox, respectively.

\subsection{Gödel's Incompleteness Theorem}\label{subsection:background-godel}

Gödel's (First) Incompleteness Theorem states that no list of axioms for a logical theory of natural number arithmetic is both \term{consistent} and \term{complete}. Consistency means the theory entails no proof of a false statement, and completeness means that the theory entails a proof of every true statement. Examples of logical theories of natural number arithmetic include Peano Arithmetic and Robinson Arithmetic. \textbf{In this paper, we restrict our focus to theories axiomatised over classical logic}.

\parsplit

\term{Gödel numbering} is a construction instrumental to Gödel's proof. Gödel numbers encode logical statements about arithmetic, e.g., sentences such as \saymath{2 > 3} or predicates such as \saymath{x > 3}. Because Gödel numbers --- being numbers --- are thus part of Arithmetic Logic, statements about Gödel numbers may be interpreted as statements about statements of Arithmetic Logic. Gödel's Incompleteness Theorem exploits the existence of a statement that reasons about its own provability via its own Gödel number. For an example of Gödel Numbering, see Nagel and Newman~\cite{nagel-newman}.

\parsplit

In Arithmetic Logic, a natural numbers $n \in \N$ is represented by formulae called the \term{numeral} of $n$. A numeral is typically defined using a successor symbol $\textbf{S}$ and a zero numeral $\textbf{0}$, i.e., the numeral of $n$ is $\textbf{SS}\dots\textbf{S}\textbf{0}$ with $n$ copies of $\textbf{S}$. Given a formula $F$, we write $\Obar{F}$ for the numeral of $F$'s Gödel number $G(F)$, i.e., with $G(F)$ copies of $\textbf{S}$.

\parsplit

Many proofs of Gödel's Incompleteness Theorem leverage the following intermediate result known as \term{The Diagonalisation Lemma}, developed by Carnap shortly after Gödel's original proof~\cite{carnap}.
\parsplit
\begin{lemma}[The Diagonalisation Lemma]\label{lemma:diag-lemma}
    For any predicate $Q(x)$ in a theory of Arithmetic Logic $\Theory$ there exists a sentence $\C$ such that $Q(\Obar{C})$ and $C$ are logically equivalent, i.e.: $\Theory \vdash Q(\Obar{C}) \leftrightarrow C$.
\end{lemma}

The key predicates of interest are derived from \term{provability predicates}, which are constructed by encoding proofs as Gödel numbers~\cite{nagel-newman}. Specifically, we are able to construct a predicate $Proof(y, x)$, which corresponds to the statement \say{$y$ is the Gödel number of a proof of a statement whose Gödel number is $x$}. Hence, the most basic provability predicate is $Provable(x) \coloneqq \exists y\ Proof(y, x)$. $Provable(x)$ has a negated form $\neg Provable(x)$, which corresponds to $\forall y:\ \neg Proof(y, x)$.

\parsplit

Applying the Diagonalisation Lemma to $\neg Provable(x)$ yields a sentence $\G$ that is logically equivalent to $\neg Provable(\Obar{\G})$. In other words, $\G$ is a sentence that appears to be \textit{true if and only if it is not provable}. The sentence $\G$ is typically called a \term{Gödel sentence} of the theory. In a theory of Arithmetic Logic, the mutual exclusivity of consistency and completeness (i.e., Gödel's Incompleteness Theorem) \textit{appears} to immediately follow from $\G$'s existence. However, this either requires refinement of the provability predicate (e.g., \say{Rosser's Trick} \cite{rosser}) or restricting to theories of Arithmetic Logic known as $\omega$-\term{consistent} theories as Gödel originally did and as we will in this paper. See~\cite[p.207]{kleene-metamathematics} for a definition of $\omega$-consistency. In this paper, we restrict to $\omega$-consistent theories $\Theory$ only to guarantee that for any proposition $Q(x)$: $\Theory \vdash Provable(\Obar{\neg Q(x)}) \rightarrow \neg Provable(\Obar{Q(x)})$.

\parsplit

\begin{theorem}[Gödel's (First) Incompleteness Theorem]\label{theorem:godel-classic}
    No $\omega$-consistent theory is complete.
\end{theorem}
\noindent See~\cite{smorynski,smullyan-godel,peter-smith-tarski} for comprehensive derivations and proofs of the theorem.

\subsection{Algebraic Logic}\label{subsection:background-algebraic-logic}

The formulation of Gödel's Incompleteness Theorem in our results utilises a construction on a logical theory known as its \term{Lindenbaum Algebra}. A Lindenbaum Algebra is a set of equivalence classes of logical formulae, where two formulae are equivalent if and only if they are logically equivalent. Reasoning about equality of elements in a Lindenbaum Algebras (in place of logical equivalence) allows us to leverage many tools from abstract algebra. This approach has been utilised in various incomputability proofs by Yanofsky~\cite{yanofsky}.

\parsplit

We begin by noting that for a logical theory $\mathcal{T}$ with symbols such as $\wedge$, $\vee$, $\neg$, propositional variables, free variables, etc., one can generate the set of all possible well-formed formulae of $\mathcal{T}$ using those symbols. We then define Lindenbaum Algebras on these formulae as follows.

\parsplit

\begin{definition}[Lindenbaum Algebras]\label{definition:lindenbaum-algebra-def}
    Given a logical theory $\mathcal{T}$ and $n \in \N$, we write $\F_n$ to denote the set of formulae with $0$ up to $n$ free variables. The \term{Lindenbaum Algebra} $\L_n$ of $\F_n$ is the set of equivalence classes of formulae in $\F_n$, where two formulae $f, g \in \F_n$ satisfy $f = g$ in $\L_n$ if and only if they have the same arity and are logically equivalent in the theory $\mathcal{T}$, meaning:
    \begin{align*}
         & \Theory \vdash C \leftrightarrow D                   &  & \text{for propositions $C, D \in \F_0$}     \\
         & \Theory \vdash \forall y:\ A(y) \leftrightarrow B(y) &  & \text{for predicates $A(x), B(x) \in \F_1$}
    \end{align*}
    etc. Where $\leftrightarrow$ is the \textit{if and only if} connective in the Theory $\Theory$.
\end{definition}

\parsplit

While logical equivalence can be expressed as equality on Lindenbaum Algebras $\L_n$ interpreted as mere sets, other aspects of logic correspond to order-theoretic and algebraic structures of Lindenbaum Algebras.

\parsplit

In terms of order theory, we observe that the set $\L_n$ ordered by the implication relation is a partial order, i.e., $f \leq g$ in $\L_n$ if and only if $\Theory \vdash f \rightarrow g$. Importantly, $\L_n$ has a bottom element $\bot \in \L_n$ or \textit{false}, which is logically equivalent to all contradictions such as $f \wedge \neg f$. The fact that $\bot$ is a bottom element, in other words, that a contradiction implies anything is known as the \textit{principle of explosion}. $\L_n$ also has a top element $\top \in \L_n$ or \textit{truth}, logically equivalent to all tautologies such as $f \vee \neg f$, and a top element because a tautology is true for any assumptions considered.

\parsplit

In terms of abstract algebra, we observe that applying logical connectives $\wedge$ (respectively $\vee$) to two (equivalence classes of) formulae corresponds to the operation of taking their greatest lower (respectively least-upper) bound in $\L_n$ with respect to implication. Likewise, negating a formula corresponds to taking its complement (in the order-theoretic sense) in $\L_n$. The combination of the set $\L_n$ and certain collections of these operations corresponds to well-known algebraic structures. For example, $(\L_n, \wedge)$ is a meet semi-lattice, and for classical logic, $(\L_n, \wedge, \vee, \neg, \bot, \top)$ is a Boolean algebra. The association of orders and algebras to different theories of logic in this way comprises a field known as Algebraic Logic.

\subsection{Social Choice Theory}\label{subsection:background-arrow}

\finestsection{Weak and Strict Orders}

Arrow's Impossibility Theorem concerns the aggregation of \term{weak orders}, i.e., transitive and complete relations. A canonical example of which is a preferential voting ballot, wherein an individual (vote) is a ranking of alternatives from most to least preferred. Weak orders permit tied rankings (i.e., \term{indifference}) between alternatives. We use the term \term{strict order} to refer to a weak order without indifference.

\parsplit

\noindent A weak order on a fixed set of alternatives $\A$ can be represented by relation symbols $\prec, \sim$ and $\preceq$ as follows:
\begin{itemize}
    \item $a \sim b$ for \term{indifference} between $a$ and $b$.
    \item $a \prec b$ for $a$ being \term{strictly preferred} to $b$  (i.e., $b \nprec a$ and $a \nsim b$).
    \item $a \preceq b$ for $a$ being \term{weakly preferred} to $b$, i.e., either $a \prec b$ or $a \sim b$ holds.
\end{itemize}

\noindent Conversely, the axioms for a weak order are correspondingly:
\begin{description}
    \item[\textbf{Transitivity:}] $\forall a,b,c \in \A$: if $a \preceq b$ and $b \preceq c$ then $a \preceq c$.
    \item[\textbf{Completeness:}] $\forall a,b \in \A$: one of $a \prec b$, $b \prec a$ or $a \sim b$ hold.
\end{description}

Moreover, weak orders may be written as a chain of the symbols $\prec$, $\sim$ and $\preceq$. For example, If $\A = \{a,b,c\}$, $a \prec b \sim c$ denotes the weak order consisting of $a \prec b$, $b \sim c$ and $a \prec c$ (by transitivity). Strict orders are chains consisting entirely of $\prec$, e.g., $c \prec a \prec b$ denotes the strict order consisting of $c \prec a$, $a \prec b$ and $c \prec b$.

\parsplit

\finestsection{Social Welfare Functions}

\noindent We conclude this section by informally summarising Arrow's Impossibility Theorem.

\parsplit

Given a fixed number $N \in \N$ of individuals, a \term{profile} is an $N$-tuple of weak orders. An example of a profile is an election, i.e., a tuple containing a single ballot from each individual. Note, each individual has a fixed index in the tuple across profiles. A \term{Social Welfare Function} is a function from a set of \textit{valid} profiles to a single aggregate weak order. Invalid profiles are those that would otherwise fail to aggregate to a weak order, e.g., by aggregating to a contradictory preference-cycle.

\parsplit

\begin{definition}\label{definition:fairness-conditions}
    A Social Welfare Function satisfies:
    \begin{itemize}
        \item \textbf{Unrestricted Domain}: If all profiles are valid with respect to it.

        \item \textbf{Unanimity}: If all individuals sharing a strict preference of $a$ over $b$ implies the aggregate does too.

        \item \textbf{Independence of Irrelevant Alternatives (IIA)}: The outcome of aggregation with respect to alternatives $a$ and $b$ only depends on the individual preferences with respect to $a$ and $b$.

        \item \textbf{Non-Dictatorship}: There is no individual such that irrespective of the profile, that individual's strict preferences are always present in the aggregate outcome. If this condition fails we say the Social Welfare Function has a \term{Dictator}\footnote{A Social Welfare Function can have at most one Dictator or else multiple Dictators could not disagree with one another.}.
    \end{itemize}
\end{definition}

\begin{theorem}[Arrow's Impossibility Theorem]
    If a Social Welfare Function on at least 3 alternatives and 2 individuals satisfies Unrestricted Domain, Unanimity and IIA then it must have a Dictator.
\end{theorem}

\noindent For examples of standard (combinatorial) proofs of Arrow's Impossibility Theorem see~\cite{geanakoplos,arrow-one-shot}.

\parsplit

Dual to a Social Welfare Function having a dictator at individual $i$ is the Social Welfare Function having a \term{vetoer} at $i$, defined as follows~\cite{blau-veto}.
\parsplit
\begin{definition}\label{definition:vetoer}
    A Social Welfare Function has a \term{vetoer} if there is an individual $i$ such that irrespective of the profile, that individual's strict preferences are never contradicted in the aggregate outcome. In other words, if in a profile, individual $i$ contains a strict preference $a \prec b$ then the aggregate weak order never contains the strict preference $b \prec a$.
\end{definition}

\noindent We combine these two conditions as follows:

\parsplit

\begin{definition}\label{definition:strong-dictator}
    A Social Welfare Function has a \term{strong dictator} if and only if there is an individual $i$ that is both a dictator and a vetoer. In other words the aggregate has a strict preference $a \prec b$ if and only if the individual $i$ does.
\end{definition}

\subsection{Condorcet's Paradox}\label{subsection:condorcet-paradox}

Condorcet's Paradox refers to phenomena where a voting system on 3 or more alternatives cannot guarantee winners that are always preferred by a majority of voters. A canonical example of a Condorcet Paradox is the observation that for the profile specified by Table \ref{table:condorcet-plain}, pairwise majority voting cannot decide a winner lest it aggregates to a contradictory preference cycle. In other words, to aggregate that profile to a weak order, there must be an aggregate preference $x \prec y$ that is only shared by a minority of individuals.

\parsplit

\begin{table}[ht!]
    \centering
    \begin{tabularx}{0.95\textwidth}{|p{1.8in}|X|X|X|}
        \hline
        \diagbox{Ranking}{Individual} & \textbf{1} & \textbf{2} & \textbf{3} \\\hline
        \textbf{1}                    & a          & b          & c          \\\hline
        \textbf{2}                    & b          & c          & a          \\\hline
        \textbf{3}                    & c          & a          & b          \\\hline
    \end{tabularx}
    \caption{A Profile on 3 voters and 3 alternatives $\{a,b,c\}$ that under pairwise majority voting, aggregates to a contradictory preference cycle.}
    \label{table:condorcet-plain}
\end{table}

To see this, consider three individuals voting on 3 alternatives $\{a,b,c\}$, and consider pairwise majority voting as our Social Welfare Function. Pairwise majority voting is defined by ranking alternatives $x \prec y$ if more voters strictly prefer $x$ to $y$ than $y$ to $x$, and $x \sim y$ if there is a tie. If we apply this rule to the profile defined by Table \ref{table:condorcet-plain}, we find that the majority of individuals strictly prefer $a$ to $b$ (individuals 1 and 3) as well as $b$ to $c$ (individuals 1 and 2), and $c$ to $a$ (individuals 2 and 3). Thus, aggregation yields a preference cycle $a \prec b \prec c \prec a$, which is contradictory given the requirement that aggregated preferences are transitive.

\parsplit

It is a simple exercise to verify pairwise majority voting satisfies Unanimity, IIA and Non-Dictatorship, but as we have seen, may violate Unrestricted Domain. In fact, all Social Welfare Functions satisfying the same constraints conditions violate Unrestricted Domain due to the existence of profiles that aggregate to preference cycles. In other words, Arrow's Impossibility Theorem is a generalisation of Condorcet's Paradox.

\parsplit

\noindent This result was first proven by D'Antoni in the special case, where all preferences are strict~\cite{dantoni}, Livson and Prokopenko later generalised D'Antoni's methodology to prove Arrow's Impossibility Theorem in full as a generalisation of Condorcet's Paradox~\cite{paper0-arxiv}. In both cases the methodology centres around extending the codomain of Social Welfare Functions to a set where each element is either a weak order or a preference cycle. This is achieved with ternary data, where for alternatives $a$ and $b$: the value 0 corresponds to to $a \prec b$, 1 for $b \prec a$ and a third value $\indiff$ for indifference: $a \sim b$. For example, for alternatives $a,b,c$ ordered as such, the tuple $(\indiff, 0, 1)$ corresponds to the weak order $a \sim b \prec c$, and the tuples $(0,0,0)$, $(1,1,1)$ and $(\indiff, 1, \indiff)$ correspond to the preference cycles $a \prec b \prec c \prec a$, $c \prec b \prec a \prec c$ and $c \prec a \sim b \sim c$, respectively. When preference-cycles arise under the assumption of transitivity, they are called \term{contradictory preference cycles} for emphasis.

\section{Results}\label{section:results}

In this section, we develop our general theory. For every new general definition or result, we provide a corresponding instantiation to both Arithmetic Logic and Social Choice Theory.

\subsection{\subsectiontexta}\label{subsection:encodings}

A core component of our general theory of \term{Self-Reference Systems} is an encoding function $\enc: \E \rightarrow \C$ from a set of \term{expressions} to a set of \term{constants}. In this section, we define encodings in the fields of Arithmetic Logic and Social Choice Theory that will be used throughout the remainder of this paper. See Appendix \ref{appendix:guide} for a table summarising the contents of each result section.

\parsplit

\finestsection{Arithmetic Logic}

\noindent We first note that for this paper's purposes, it suffices to reason about formulae with 0 free variables (propositions) or 1 free variable (predicates). Then, we recall that given a set $\F_n$ of formulae with 0 up to $n$ free variables, the Lindenbaum Algebra $\L_n$ is the set of equivalence classes of formulae in $\F_n$ with respect to logical equivalence in the theory (Definition \ref{definition:lindenbaum-algebra-def}). Given any Gödel Numbering on formulae $G: \F_1 \rightarrow \N$, we can use $G$ to define a Gödel Numbering $\NGodelvar: \L_1 \rightarrow \N$, i.e., on equivalence classes of formulae.

\parsplit

\begin{definition}\label{definition:enc}
    Given a Gödel Numbering $G: \F_1 \rightarrow \N$, we define $\NGodelvar: \L_1 \rightarrow \N$ by mapping the equivalence class of a formula $f$ to $G(f')$, where $f'$ is the shortest formula among those logically equivalent to $f$.

    \parsplit

    Analogously to the shorthand $\Obar{f} \coloneqq \underline{G(f)}$, i.e., the numeral of $G(f)$ (see Definition \ref{definition:lindenbaum-algebra-def}), we introduce $\DObar{f} \coloneqq \underline{\NGodel{f}} = \underline{G(f')} = \Obar{f'}$, i.e., the Gödel numeral of the shortest formula $f'$ logically equivalent to $f$.
\end{definition}

\parsplit

\noindent Importantly, $\NGodelvar$ is well-defined in the following sense:

\parsplit

\begin{restatable}{proposition}{godelnumbering}\label{proposition:well-defined}
    The function $\NGodelvar: \L_1 \rightarrow \N$ defined as in Definition \ref{definition:enc} satisfies:
    \begin{enumerate}
        \item $\forall f, g \in \F_1$: $f$ and $g$ are logically equivalent if and only if $\NGodel{f} = \NGodel{g}$.
        \item If $f, g \in \F_1$ are logically equivalent then for any predicate $B(x)$ so are $B(\DObar{f})$ and $B(\DObar{g})$
        \item $\NGodelvar$ is computable function when $G$ is.
    \end{enumerate}
\end{restatable}
\begin{proof}
    See Appendix \ref{appendix:godel-numbering}.\\
\end{proof}

\parsplit

\finestsection{Social Choice Theory}

\noindent In Section \ref{subsection:condorcet-paradox} we discussed D'Antoni~\cite{dantoni}, and Livson and Prokopenko's~\cite{paper0-arxiv}'s work that demonstrates that all Social Welfare Functions satisfying the constraints of Arrow's Impossibility Theorem other than Unrestricted Domain, necessarily fails to satisfy it lest a profile exists that aggregates to a contradictory preference cycle. Moreover, this is shown by studying Social Welfare Functions with an extended codomain, where each element is either a weak order or a preference cycle.

\parsplit

For the purposes of this paper, we argue that because transitivity of preferences is assumed in the Arrovian framework --- so that all preference cycles are contradictory --- all such contradictory preference cycles are in fact equivalent in the following manner that is analogous to all contradictions being equivalent in Algebraic Logic.

\parsplit

Preference cycles such as $a_1 \prec a_2 \prec a_3 \prec a_1$, combined with transitivity of preferences are contradictory because the combination of the two implies that all strict preferences $a_n \prec a_m$ and $a_m \prec a_n$ simultaneously hold. Analogously, in $\L_0$, the proposition \textit{false} (i.e., $\bot$) is logically equivalent to all contradictions such as $C \wedge \neg C$. Thus, we first formalise contradictory preference cycles as all being equivalent, and then as a bottom element in an Algebraic Logic-like ordering.

\parsplit

\begin{definition}\label{definition:complete-condorcet-paradox}
    Let $\A$ be a fixed set of alternatives and $N \geq 2$ individuals. Denoting $\P$ as the set of weak orders and $\PN$ as the set of $N$-individual profiles on $\P$, an \term{(extended) Social Welfare Function} is a function $\PN \rightarrow \PC$ where $\PC \coloneqq \P \cup \{\c\}$ for an element $\c \notin \P$.

    \parsplit

    The element $\c$, representing all equivalent, contradictory preference-cycles on $\A$, will be referred to as \term{the contradictory preference cycle}. We call elements of $\PC$ \term{preference relations}. I.e.,
    \begin{equation*}
        \text{Preference Relations} = \text{Weak Orders} + \text{The Contradictory Preference Cycle}
    \end{equation*}
\end{definition}

To realise $\c$ as a bottom element of an appropriate ordering of $\P$, we note that $\P$ can be ordered by a strictness relation. Specifically, a weak order $r$ is stricter than a weak order $s$ (denoted $r \leq s$) if $r$ has at least all the strict preferences of $s$, i.e., for any alternatives $a, b \in \A$: $a \prec b$ (i.e., strictly) in $s \implies a \prec b$ in $r$. Note, the partial order $(\P, \leq)$ has a top element given by the weak order indifferent on all alternatives, denoted $\i \in \P$. For example, for $\A = \{a,b,c\}$: $\i$ corresponds to $a \sim b \sim c$. Then, we extend $\P$ to include the contradictory preference cycle as follows.

\parsplit
\begin{definition}\label{definition:extending-order}
    Given the strictness ordering $(\P, \leq)$ on the set $\P$ weak orders on alternatives $\A$. We extend $\leq$ to $\PC \coloneqq \P \cup \{\c\}$ by adding the minimal number of relations to satisfy $\forall r \in \PC$: $\c \leq r$, i.e., for $\c$ to be the bottom element of $(\PC, \leq)$.
\end{definition}

The constraints of Arrow's Impossibility Theorem can now be defined in terms of Social Welfare Functions defined as functions $\Welfare: \PN \rightarrow \PC$.

\parsplit

\begin{definition}
    A Social Welfare Function $\Welfare: \PN \rightarrow \PC$ satisfies \term{Unrestricted Domain} holds if and only if $im(w) = \P$.
\end{definition}

\parsplit

\begin{definition}\label{definition:dictator-generalised}
    A Social Welfare Function $\Welfare: \PN \rightarrow \PC$ \term{has a dictator at (individual) $i$} if and only if $\forall p \in \PN$:
    \begin{enumerate}
        \item $\Welfare(p) = \c \implies p_i = \i$
        \item If $p_i \neq \i$ then $\Welfare(p) \leq p_i$
    \end{enumerate}
\end{definition}
In other words, if a Social Welfare Function has a dictator at $i$ then as long as individual $i$ has any strict preferences, not only must the aggregate outcome not contradict the dictator, the aggregate outcome must not be a contradictory preference cycle.

\parsplit

These definitions correspond to those given in \cite{paper0-arxiv,dantoni} with the identification of all preference cycles as a single element $\c$. These works also have corresponding definitions for IIA and Unanimity for extended Social Welfare Functions, although we do not explicitly use them in this paper.

\parsplit

\noindent Returning to the original task of defining an encoding function to instantiate our general theory to Social Choice Theory, we proceed with functions $\Welfare: \PCN \rightarrow \PC$.

\begin{note}\label{note:welfare-abuse}
    We will refer to every such function $\Welfare$ as a Social Welfare Function when restricting its domain from $\Welfare$ to $\PN$ --- allowing any behaviour of $\Welfare$ outside of $\PN$.
\end{note}

\parsplit

To conclude this section, we define additional operations on $\PC$ that are used to establish impossibility in Social Choice Theory, they are more formally defined using relational algebra in Propositions \ref{proposition:bounds} and \ref{proposition:missing-bounds}.

\parsplit

\noindent Firstly, we observe for any two weak orders $r, s \in \P$, there always exists a least upper bound $r \vee s$, which is the strictest weak order that is no stricter than either of $r$ and $s$. The least upper bound $\vee$ always exists over $\PC$ too as $\c \vee r = r \vee \c = r$ for every $r \in \PC$ recalling that $\c$ is the bottom element of $\PC$ by definition.

\parsplit

\begin{example}
    If $r$ represents $a \prec b \prec c$, and $s$ represents $b \prec a \prec c$ then $r \vee s$ represents $a \sim b \prec c$.
\end{example}

\parsplit

Dually, for any two preference relations $r, s \in \PC$, there always exists a greatest lower bound $r \wedge s$, which is the least strict preference relation, stricter than both of $r$ and $s$.

\parsplit

\begin{example}
    If $r$ represents $a \sim b \prec c$ and $s$ represents $a \prec b \sim c$ then $r \wedge s$ represents $a \prec b \prec c$.
\end{example}

\parsplit

Logically, it is also useful to think of the following \term{join} of two preference relations, which is more specifically the strictest weak order that is \textit{compatible} with all information in two underlying preference relations in the following sense:

\parsplit

\begin{definition}\label{definition:veebar}
    Given $\PC$ as in Definition \ref{definition:extending-order}, we define the \term{join} binary operation $\join: \PC \times \PC \rightarrow \PC$ by mappings:
    \begin{equation}
        r \join s = \begin{cases}
            r \vee s & \text{If $r \in \P$ and $s \in \P$}                           \\
            \i       & \text{Otherwise, i.e., if $r = \c$ or $s = \c$, exclusively.}
        \end{cases}
    \end{equation}
\end{definition}

\noindent Intuitively, $\forall r \in \PC$: $r + \c = \i$ because $\c$ \textit{containing} all strict preferences $a_n \prec a_m$ means that the only valid preference relation (i.e., weak order) that can be compatible with $\c$ is $\i$.

\parsplit

This definition of $\join$ has an important correspondence with our generalised definition of dictators in Definition \ref{definition:dictator-generalised}.
\begin{restatable}{proposition}{dictatorequivalence}
    \label{proposition:dictator-equivalence}
    A Social Welfare Function $\Welfare: \PN \rightarrow \PC$ has a dictator at $i$ if and only if $\forall p \in \PN$: $\Welfare(p) \join p_i = p_i$.
\end{restatable}
\begin{proof}
    See Appendix \ref{appendix:result-seciton-proofs}.\\
\end{proof}
\noindent Dually, $\forall p \in \PN$: $\Welfare(p) \wedge p_i = p_i$ if and only if $i$ is a vetoer.

\parsplit

Because $\vee$ and $\wedge$ compute least upper bounds and greatest lower bounds respectively, they are meet and join semi lattices that satisfy $r \leq s \iff r \vee s = s \iff r \wedge s = r$~\cite{lattice-order-theory}.  It is a simple exercise to verify $\PC$ is an orthocomplemented lattice with negation defined as follows:

\begin{definition}
    Negation $\neg: \PC \rightarrow \PC$ is defined by the mappings $\neg c = \i$ and $\neg i = \c$, and for any other weak order $r$, $\neg r$ is the weak order with $r$'s strict preferences flipped.
\end{definition}

\begin{example}
    If $r$ represents $a \sim b \prec c$ then $\neg r$ represents $c \prec a \sim b$.
\end{example}

\subsection{\subsectiontextb}\label{subsection:first-def}

In this section, we define the fundamental object of our general theory: Self-Reference Systems and provide examples of Self Reference Systems in Arithmetic Logic and Social Choice Theory used in results throughout.

\parsplit
\begin{definition}\label{definition:self-reference-system}
    A \term{Self-Reference System} $(\enc, \app)$ is a combination of:
    \begin{itemize}
        \item A set $\C$ of \term{constants}
        \item A set $\E$ of \term{expressions}
        \item An \term{encoding function} $\enc: \E \rightarrow \C$
        \item An \term{application function} $\app: \E \times \C \rightarrow \E$
    \end{itemize}
\end{definition}

\begin{note}\label{note:ast-shorthand}
    To reduce bracketing, we introduce a binary operation $\ast$ on $\E$ defined by $e \ast f \coloneqq \app(e, \enc(f))$.
\end{note}

\parsplit

\noindent The following examples Self-Reference Systems we will motivate our use of the phrase \say{Self-Reference}. In short, Self-Reference typically arises out of applying expressions to their own encoding, i.e., expressions of the form $e \ast e$.

\parsplit

That said, the phrase \say{Self-Reference} is used primarily for motivation; there are of course, trivial examples of Self-Reference Systems, e.g., setting $\C = \E$, the identity function $id_\E$ for encoding and the left projection $\pi_0$ for application.

\parsplit
\begin{example}[Self-Reference Systems in Arithmetic Logic]\label{example:godel-1}
    For all Arithmetic Logic examples in this paper, fixing a theory of Arithmetic Logic $\Theory$ and a Gödel Numbering, we define a Self-Reference System $(\NGodelvar, \app)$ by taking:
    \begin{itemize}
        \item The Natural Numbers $\N$ for constants.
        \item The Lindenbaum Algebra $\L_1$ for expressions (i.e., predicates and propositions, see Definition \ref{definition:lindenbaum-algebra-def}).
        \item Encoding $\NGodelvar: \L_1 \rightarrow \N$ mapping (the equivalence classes of a) formula to the Gödel numeral of the shortest equivalent formula (see Definition \ref{definition:enc}).
        \item Application $\app: \L_1 \times \N \rightarrow \L_1$ defined by:
              \begin{align*}
                  \app(B(x) & , n) \coloneqq B(\underline{n}) &  & \text{For predicates $B(x) \in \L_1$ and numeral $\underline{n} \in \L_0$} \\
                  \app(D    & , m) \coloneqq D                &  & \text{Otherwise, i.e., for propositions $D \in \L_0 \subset \L_1$}
              \end{align*}
    \end{itemize}
    In this example, the application function $\app$ is analogous to substitution.
    In terms of self-reference, for any predicate $B(x) \in \L_1$, the formula $B(x) \ast B(x) = \app(B(x), \NGodel{B(x)}) = B(\DObar{B(x)})$ can be considered self-referential. This is because the predicate $B(\DObar{B(x)}) = B'(\Obar{B'(x)})$, for $B'(x)$ the shortest formula logically equivalent to $B(x)$, which thus \textit{refers} to its own Gödel numeral.
\end{example}

\parsplit
\begin{example}[Self-Reference Systems in Social Choice Theory]\label{example:arrow-1}
    Recall the set $\PC$ defined as the set of weak orders on a fixed set of alternatives and the contradictory preference cycle (see Definition \ref{definition:extending-order}). Fixing an individual $i$, we define three different Self-Reference Systems $(\Welfare, \join_i)$, $(\Welfare, \wedge_i)$ and $(\Welfare, \Omega_i)$ all with:
    \begin{itemize}
        \item Individual preference relations: $\PC$ for constants.
        \item Profiles (i.e., tuples) of $N$ preference relations: $\PCN$ for expressions.
        \item A (Social Welfare) function $\Welfare: \PCN \rightarrow \PC$ for encoding (see Note \ref{note:welfare-abuse}).
    \end{itemize}
    The first Self-Reference System $(\Welfare, \join_i)$ has application function:
    \begin{align*}
        \join_i(\ (p_1, \dots, p_i, \dots, p_N),\ r\ ) & \coloneqq (p_1, \dots, p_i \join r, \dots, p_N)
    \end{align*}
    (see Definition \ref{definition:veebar}). In terms of self-reference, consider expressions $p \ast p$, which at coordinate $i$ combines $p_i$ (individual $i$) with the encoding (aggregate) $\Welfare(p)$, i.e., $p_i \join \Welfare(p)$. We are interested in cases where there is a coupling between group preferences and an individual's preference, despite $\Welfare(p)$ being completely determined by $p$. For example, by Proposition \ref{proposition:dictator-equivalence} there is a dictator at $i$ when the following is always satisfied:
    \begin{equation*}
        \Welfare(p_1, \dots, p_i, \dots, p_N) \leq p_i \quad \text{or equivalently} \quad p_i \join \Welfare(p_1, \dots, p_i, \dots, p_N) = p_i
    \end{equation*}

    Here, the presence of $p_i$ on both sides represents the coupling between \textit{expression} and \textit{encoding}, which can be illustrated by telescoping at the $i$-th coordinate in a self-referential fashion as:
    \begin{equation*}
        \Welfare(p_1, \dots, \Welfare(p_1, \dots, \Welfare(\dots), \dots, p_N), \dots, p_N) \leq p_i
    \end{equation*}

    Two other two useful Self-Reference Systems $(\Welfare, \wedge_i)$ and $(\Welfare, \Omega_i)$ have application functions:
    \begin{align*}
        \wedge_i(\ (p_1, \dots, p_i, \dots, p_N),\ r\ ) & \coloneqq (p_1, \dots, p_i \wedge  r, \dots, p_N) \\                                                               
        \Omega_i(\ (p_1, \dots, p_i, \dots, p_N),\ r\ ) & \coloneqq (\c, \dots, p_i \wedge  r, \dots, \c)
    \end{align*}
\end{example}

\noindent Now that we have our primary examples of Self-Reference Systems in Arithmetic Logic and Social Choice Theory, we may proceed to define additional relevant properties that Self-Reference Systems may satisfy.

\subsection{\subsectiontextc}\label{subsection:fp-property}

For any Self-Reference System, we can define a fixed-point property, which when satisfied in a certain manner for the Self-Reference Systems of Arithmetic Logic (Example \ref{example:godel-1}), implies the fixed-point condition of the standard Diagonalisation Lemma (Lemma \ref{lemma:diag-lemma}). Moreover, the fixed-point property being satisfied in a certain manner for the Self-Reference Systems of Social Choice Theory (Example \ref{example:arrow-1}), is equivalent to that Social Welfare Function having a dictator.

\parsplit

We begin by defining the fixed-point property for Self-Reference Systems in general. Then, we restate the Diagonalisation Lemma in Arithmetic Logic, and the definition of a Dictator in Social Choice Theory in terms of statements about the fixed-point property holding for particular Self-Reference Systems.

\parsplit
\begin{definition}\label{definition:fixed-point-property}
    A Self-Reference System $(\enc, \app)$ satisfies the \term{fixed point property} for an expression $e \in \E$ if there exists an $f \in \E$ such that $e \ast f = f$, i.e.,  $\app(e, \enc(f)) = f$.
\end{definition}

\parsplit
\begin{proposition}[The Diagonalisation as the Fixed-Point Property]\label{proposition:diagonalisation-lemma-instance}
    If the Self-Reference System $(\NGodelvar, \app)$ of Example \ref{example:godel-1} satisfies the fixed point property by for all expressions in $\L_1$ by fixed-points in $\L_0 \subset \L_1$, then the standard Diagonalisation Lemma holds.
\end{proposition}
\begin{proof}
    Recall by the definition of $\NGodelvar$ (Definition \ref{definition:enc}) that for a formula $f$, we write $f'$ to denote the shortest formula logically equivalent to $f$. If for an arbitrary predicate $Q(x) \in \L_1$ there is a sentence $C \in \L_0$ such that $C = Q(x) \ast C = Q(\DObar{C})$, then by the definition of Lindenbaum Algebras this implies that $\Theory \vdash Q(\Obar{C'}) \leftrightarrow C'$, thus satisfying the standard Diagonalisation Lemma by the arbitrariness of $Q(x)$.\\
\end{proof}

\noindent The converse to Proposition \ref{proposition:diagonalisation-lemma-instance} does not necessarily hold. A proof that $(\NGodelvar, \app)$ satisfies the fixed-point property as above will be demonstrated in Section \ref{subsection:abstract-diag}.

\parsplit
\begin{proposition}[Dictators and the Fixed-Point Property]\label{proposition:dictator-as-fixed-point}
    Given the Self-Reference System $(\Welfare, \join_i)$ of Example \ref{example:arrow-1}, the Social Welfare Function corresponding to $\Welfare$ has a dictator at individual $i$ if and only if for every valid profile $p \in \PN$: $(\Welfare, \join_i)$ satisfies the fixed-point property by $p$ itself, i.e., $p \ast p = p$.
\end{proposition}
\begin{proof}
    Given an arbitrary $p = (p_1, \dots, p_i, \dots, p_N) \in \PN$, $p \ast p = p$ occurs if and only if:
    \begin{equation}
        (p_1, ..., p_i \join \Welfare(p), ..., p_N) = (p_1, ..., p_i, ..., p_N)
    \end{equation}
    This occurs if and only if $p_i \join \Welfare(p) = p_i$. But by the arbitrariness of $p$ this is equivalent to $\Welfare$ having a dictator at $i$ by Proposition \ref{proposition:dictator-equivalence}.\\
\end{proof}

We have shown that key components of incompleteness in Arithmetic Logic (The Diagonalisation Lemma) and impossibility in Social Choice Theory (The Existence of a Dictator) correspond to the fixed-point property being satisfied in a particular manner for particular Self-Reference Systems.

\parsplit

However, to prove that the fixed-points in question exist in Arithmetic Logic and to demonstrate that fixed-points of the same type exist in Social Choice Theory, we need to develop an abstract theory of Diagonalisation (Section \ref{subsection:abstract-diag}). In order to develop that theory, we must first define an important subtype of Self-Reference Systems called Embeddable Self-Reference Systems, the subject of the next section.

\subsection{\subsectiontextd}\label{subsection:embeddable}

To prove predicates in the Self-Reference Systems underlying Arithmetic Logic satisfy the fixed-point property (Proposition \ref{proposition:diagonalisation-lemma-instance}), we derive a result called the \textit{Abstract Diagonalisation Lemma} (Theorem \ref{theorem:diag-abstract}). In order to state and prove that result, we need to define a special type of Self-Reference System called an Embeddable Self-Reference System.

\parsplit

A Self-Reference System is Embeddable when --- informally --- \textit{there is a way to calculate everything at the expression level}. Or more specifically, a Self-Reference System is Embeddable when there is an associative, binary \term{composition} operation on expressions and an \term{embedding} operation from constants to expressions such that application (of an expression to a constant) is equivalent to composition with an embedding. We begin by defining Embeddable Self-Reference Systems and then demonstrate that our existing examples of Self-Reference Systems (Examples \ref{example:godel-1} and \ref{example:arrow-1}) are embeddable.

\parsplit
\begin{definition}\label{definition:embeddable}
    Given an associative binary \term{composition} operation $\comp: \E \times \E \rightarrow \E$ and an \term{embedding} function $\emb: \C \rightarrow \E$, we say that a Self-Reference System $(\enc, \app)$ defined on expressions $\E$ and constants $\C$ is \term{embeddable} in $(\emb, \comp)$ if:
    \begin{equation}\label{equation:embeddable}
        \forall e \in \E,\ \forall c \in \C:\ \app(e, c) = e \comp \emb(c)
    \end{equation}
    As shorthand, instead of writing that the Self-Reference System $(\enc, \app)$ is embeddable in $(\emb, \comp)$, we write the \term{Embeddable Self-Reference System} $(\enc, \emb, \comp)$ because $\app$ can be defined by Equation (\ref{equation:embeddable}).
\end{definition}

\parsplit

\begin{definition}
    For an embeddable Self-Reference System $(\enc, \emb, \comp)$ with expressions $\E$, we write $\DObar{-}$ for the composite $\emb \circ \enc$, i.e., $\forall e \in \E$: $\DObar{e} = \emb(\enc(e))$.
\end{definition}

\parsplit

\noindent  Importantly, this means that for expressions $e, f \in \E$: $e \ast f = e \comp \DObar{f}$.

\parsplit
\begin{example}\label{example:godel-embedding}
    The Self-Reference System $(\NGodelvar, \app)$ (see Example \ref{example:godel-1}) is embeddable in $(\emb, \comp)$ for the number to numeral inclusion $\emb: \N \hookrightarrow \L_1$ (i.e., $\emb(n) = \underline{n}$) and substitution $\comp$ defined by
    \begin{align*}
         & A(x) \comp B(y) \coloneqq A(B(y)) &  & \text{For $A(x), B(x) \in \L_1$}                          \\
         & A(x) \comp C \coloneqq A(C)       &  & \text{For $A(x) \in \L_1$ and $C \in \L_0$}               \\
         & D \comp f \coloneqq D             &  & \text{For any sentence $D \in \L_0$ and any $f \in \L_1$}
    \end{align*}
    Condition (\ref{equation:embeddable}) holds because for any $A(x) \in \L_1$ and $n \in \N$: $A(x) \comp \emb(n) = A(\underline{n}) = \app(A(x), n)$, using the definition of $\app$ in Example \ref{example:godel-1}. This example shows that the embedding function $\emb$ coupled with the composition operation $\comp$ succeeds in recovering the application function $\app$ of Self-Reference System in point. This demonstrates formally that the application function is given by substitution.
\end{example}

\parsplit
\begin{definition}\label{definition:indicator}
    Given $\PC$ and $\PCN$ as per Definition \ref{definition:complete-condorcet-paradox} and an individual $i \in \{1,\dots,N\}$, we define the \term{$i^{th}$ indicator} as $\T_i \in \PCN$ as the element $(\c, \dots, \i, \dots, \c) \in \PCN$, i.e., consisting entirely of $\c$ except at the $i^{th}$ index.
\end{definition}

\begin{example}\label{example:arrow-embedding}
    For a fixed individual $i$, we define $\emb_i: \PC \rightarrow \PCN$ such that for $\bwedge$ defined on $\PN$ as coordinate-wise $\wedge$, we find for instance that $(\Welfare, \Omega_i)$ of Example \ref{example:arrow-1} is embeddable in $(\emb_i, \bwedge)$ with $\emb_i(r) \coloneqq (\c, \dots, r, \dots, \c)$:
    \begin{align*}
        \Omega_i(p, \Welfare(q)) & = (\c, \dots, p_i\ \wedge\ \Welfare(q), \dots, \c)                           \\
                                 & = (p_1\ \wedge\ \c, \dots, p_i\ \wedge\ \Welfare(q), \dots, p_N\ \wedge\ \c) \\
                                 & = p \bwedge \DObar{q}
    \end{align*}
\end{example}
Although not used in this paper, note that altering the above embedding to $\text{altEmb}(r) \coloneqq (\i, \dots, r, \dots, \i)$, $(\Welfare, \wedge_i)$ is embeddable in $(\text{altEmb}, \wedge)$. Then, defining the $\text{altComp}: \PCN \times \PCN \rightarrow \PCN$ to compute $\join$ at coordinate $i$ and $\wedge$ otherwise, we can internalise $(\Welfare, \join_i)$ in $(\text{altEmb}, \text{altComp})$.

\parsplit

When a Self-Reference Systems is Embeddable Self-Reference Systems, we can reason entirely at the level of expressions\footnote{One-typed frameworks for studying diagonalisation has features in other works, e.g., Gaifman \cite{gaifman} and Smullyan \cite{smullyan-diagonalization,smullyan-theory-formal-systems} although without explicit encoding and embedding.}. We will exploit this fact to prove the Abstract Diagonalisation Lemma in the next section.

\subsection{\subsectiontexte}\label{subsection:abstract-diag}

In this section, we derive a generalised version of the Diagonalisation Lemma for Self-Reference Systems called the \term{Abstract Diagonalisation Lemma}. In the Self-Reference Systems of Arithmetic Logic (see Example \ref{example:godel-1}), the Abstract Diagonalisation Lemma yields the standard Diagonalisation Lemma. Moreover, in Social Choice Theory, the Self-Reference System $(\Welfare, \Omega_i)$ (see Example \ref{example:arrow-1}) satisfies the Abstract Diagonalisation Lemma if and only if $\Welfare$ has a dictator at $i$.

\parsplit

This Abstract Diagonalisation Lemma exploits a new property of expressions in a Self-Reference System called \term{internalisation}, which amounts to a particular expression in $\E$ of a Self-Reference System being a \textit{code} for a function \textit{external} to the Self-Reference System (e.g., a $\E \rightarrow \E$ function). We define internalisation generally as follows:

\parsplit
\begin{definition}
    A function $\alpha: X \rightarrow Y$ is \term{internalised} by a function $\beta: Z \times X \rightarrow Y$ if $\exists\ z_\alpha \in Z$ such that:
    \begin{equation*}
        \forall x \in X:\ \alpha(x) = \beta(z_\alpha, x)
    \end{equation*}
    We also say that $z_\alpha$ \term{internalises} $\alpha$ with respect to $\beta$.
\end{definition}

\begin{note}
    Related definitions exist in other approaches, e.g., Yanofsky uses the term \say{representable} for internalisation with respect to functions $T \times T \rightarrow Y$~\cite{yanofsky}. In Category Theory, Lawvere uses the term \textit{weakly point surjective} to refer to functions $\beta: Z \times X \rightarrow Y$ such that all functions of the form $X \rightarrow Y$ can be internalised with respect to $\beta$~\cite{Lawvere}.
\end{note}

\parsplit
\begin{example}\label{example:godel-diagonaliser}
    In Arithmetic Logic, internalisation can be motivated as follows: because substituting the free variable of a predicate in $\F_1$ with a sentence in $\F_0$ produces a new formula in $\F_0$, many predicates in $\F_1$ \textit{behave} like a $\F_0 \rightarrow \F_0$ function. Moreover, many $\F_1$ elements behave like $\F_1 \rightarrow \F_1$ functions because of the following inclusions:
    \begin{center}
        \begin{tikzcd}
            \F_1 \arrow[r, "G", hook] & \N \arrow[rr, "Numeral", hook] && \F_0 \arrow[rr, "inclusion", hook] && \F_1
        \end{tikzcd} $\implies \F_0 \cong \F_1$
    \end{center}
    As such, the key internalisation that establishes the Abstract Diagonalisation Lemma is the predicate $diag(x)$ internalising the \textit{diagonal function}:
    \begin{equation*}
        \diag: \L_1 \rightarrow \L_1 \quad\text{defined by mappings}\quad B(x) \mapsto \DObar{B(\DObar{B(x})}
    \end{equation*}
\end{example}

\begin{note}\label{note:internalisablility-caveat}
    Demonstrating that $\diag(x)$ exists is highly non-trivial. Often diagonalisation of more intricate functions than $\diag$ is demonstrated instead (see Salehi~\cite{salehi-presentation, salehi-diagonal-lemma} for examples of relevant approaches). That said, given that the diagonaliser predicate $diag(x)$ exists in $\F_1$, its existence over $\L_1$ (i.e., with respect to our Gödel numbering $G': \L_1 \rightarrow \N$) follows because $G'$ is computable given $G$ is (see Proposition \ref{proposition:well-defined}).
\end{note}

We now proceed to define the diagonal function $\diag$ for Self-Reference Systems in general.

\parsplit
\begin{definition}\label{definition:composites}
    Given an Embeddable Self-Reference System $(\enc, \emb, \comp)$ with expressions $\E$, we define $\diag: \E \rightarrow \E$ by the following composites (left), defined by mappings (right):
    \begin{center}
        \begin{tikzcd}
            E \times E \arrow[r, "\ast"]                                                     & E \arrow[r, "\DObarvar"] & E &  & {(e, e)} \arrow[r, maps to]                                               & {e \ast e} \arrow[r, maps to] & \DObar{e \ast e} \\
            E \arrow[u, "\Delta"] \arrow[rru, "\diag"'] &                         &   &  & e \arrow[u, maps to] \arrow[rru, maps to] &                                        &
        \end{tikzcd}
    \end{center}
\end{definition}

\parsplit

\begin{definition}
    A Self-Reference System $(\enc, \app)$ with expressions $\E$ has a \term{diagonaliser} if it has an expression $f_\diag \in \E$ that internalises the function $\diag: \E \rightarrow \E$ with respect to $\ast$ (see Definition \ref{definition:composites}). In other words, for every expression $e \in \E$: $\diag \ast e = \DObar{e \ast e}$.
\end{definition}

\parsplit

\begin{restatable}{theorem}{abstractdiagonalisationlemma}{\textbf{(The Abstract Diagonalisation Lemma)}}\label{theorem:diag-abstract}
    If an Embeddable Self-Reference System $(\enc, \emb, \comp)$, has a diagonaliser $f_\diag$ then the fixed-point property is satisfied for all expressions.
\end{restatable}
\begin{proof}
    For an arbitrarily $d \in \E$, we define $e \coloneqq d \comp f_\diag$, and $f \coloneqq e \ast e$, and find that $d \ast f = f$ as desired. For the full calculation, see Appendix \ref{appendix:result-seciton-proofs}.\\
\end{proof}

\parsplit

\begin{theorem}[The Classic Diagonalisation Lemma]
    \label{theorem:arithmetic-logic-fixed-points}
    The Embeddable Self-Reference System $(\NGodelvar, \comp, \emb)$ of Example \ref{example:godel-embedding} has a diagonaliser and satisfies the fixed-point property for all expressions in $\L_1$ by fixed-points in $\L_0 \subset \L_1$.
\end{theorem}
\begin{proof}
    By Note \ref{note:internalisablility-caveat}, we ascertain the existence of a diagonaliser (predicate) $diag(y)$. Then, given any predicate $A(x)$, we produce a fixed point by directly following the proof of the Abstract Diagonalisation Lemma as follows. We first define the predicate $D(x)$ to be $A(x) \comp diag(y) = A(diag(y))$, i.e., so that $\forall B(x) \in \L_1$: $D(\DObar{B(x)}) = A(\DObar{B(\DObar{B(x)})})$ (see Note \ref{note:internalisablility-caveat}). Then, the fixed-point of $A(x) \in \L_1$ is $D(\DObar{D(x)}) \in \L_0$.\\
\end{proof}

\parsplit

\begin{theorem}[Dictators \& Diagonalisers]\label{theorem:dictator-as-diagonaliser}
    Given the Self-Reference System $(\Welfare, \Omega_i)$ of Example \ref{example:arrow-1}, the Social Welfare Function corresponding to $\Welfare$ has a dictator at individual $i$ if and only if the $i^{th}$ indicator $\T_i = (\c, \dots, \i, \dots, \c)$ is a diagonaliser restricted to the domain $\PN$.
\end{theorem}
\begin{proof}
    Firstly, we establish when $\T_i$ is a diagonaliser. Indeed, $p \in \PCN$ we must have that:
    \begin{equation}\label{equation:diagonaliser}
        \T_i \ast p = (\c, \dots, \i \wedge \Welfare(p), \dots, p) = (\c, \dots, p_i \wedge \Welfare(p), \dots, p) = p \ast p
    \end{equation}
    However, by Proposition \ref{proposition:dictator-equivalence}, Equation (\ref{equation:diagonaliser}) holds $\forall p \in \PN$ if and only if $\Welfare(p) = p_i\ \wedge\ \Welfare(p)$, i.e., $\Welfare(p) \leq p_i$, i.e., $\Welfare$ has a dictator at $i$.\\
\end{proof}

\begin{note}
    Because $\T_i$ is a diagonaliser only with respect to the restricted domain $\PN \subset \PCN$, we cannot guarantee $(\Welfare,\Omega_i)$ satisfies the Abstract Diagonalisation Lemma without further specification for how $\Welfare$ behaves outside $\PN$, in particular for expressions of the form $e \ast e = (\c, \dots, \Welfare(e), \dots, \c)$ for $e = \T_i \ast p = (\c, \dots, \Welfare(p), \dots, \c)$.
\end{note}

\subsection{\subsectiontextf}\label{subsection:const-resp}

In this section, we characterise incompleteness in Arithmetic Logic and impossibility Social Choice Theory with respect to a predefined subset of \term{valid elements}. A pair of expressions are \term{consistent} if holding them together (e.g., by logical conjunction) yields a valid element, and the pair is called \term{inconsistent} otherwise. \term{Consistency-respecting} expressions are those that preserve and reflect consistency with respect to application with encodings alone. The incomputability of particular consistency-respecting expressions will be shown to arise within both Gödel's Incompleteness Theorem and Arrow's Impossibility Theorem.

\parsplit

We begin with the definition of a subset of valid elements as a general construction on a meet semi-lattice. Then, we identify valid elements in Self-Reference Systems of Arithmetic Logic and Social Choice Theory.

\parsplit
\begin{definition}\label{definition:me}
    Given a set $S$ and a meet semi-lattice $(S, \wedge)$ with bottom element $\bot \in S$, a \term{subset of valid elements} of $S$ with respect to $\wedge$ is simply a subset $C \subseteq S \setminus \{\bot\}$. For any pair $s, t \in S$, we say that $s$ and $t$ are \term{consistent}, denoted $s \const t$ when $s \wedge t \in C$; they are \term{inconsistent} otherwise, denoted $s \inconst t$.
\end{definition}

\parsplit
\begin{example}[Valid Elements in Arithmetic Logic]\label{example:godel-consistent-pairs}
    In Arithmetic Logic, for propositions $\L_0$, we take all non-contradictory propositions as our subset of valid elements, i.e., $\L_0 \setminus \{\bot\}$ with respect to logical conjunction $\wedge$. In other words, $C, D \in \L_0$ are consistent (i.e., $C \const D$) when $C \wedge D \neq \bot$. Moreover, $\N$ has an equivalent subset of valid elements: $G'(\L_0 \setminus \{\bot\}) \subseteq \N \setminus \{G'(\bot)\}$, i.e., so $n \const m$ in $\N$ when they are the Gödel $n$ and $m$ are the Gödel numbers of consistent propositions.
\end{example}

\parsplit
\begin{example}[Valid Elements in Social Choice Theory]\label{example:soc-me}
    In Social Choice Theory, for preference relations $\PC$, we take weak orders as our subset of valid elements, i.e., $\P \subseteq \PC \setminus \{\c\}$ with respect to $\wedge$ (see Section \ref{subsection:encodings}). So $p \const q$ in $\PC$ when $p \wedge q \neq \c$. Then, over $N$ individuals, we take \textit{valid profiles} (i.e., of weak orders) as our subset of valid elements, i.e., $\PN \subset \PCN \setminus \{(\c,\dots,\c)\}$ with respect to $\wedge$ defined coordinate-wise on $\PCN$. This means that $p \const q$ in $\PCN$ are consistent if for every individual $i$, $p_i$ and $q_i$ are consistent to one another, i.e., $p_i \wedge q_i \neq \c$. Inconsistency here, matches the definition of inconsistent profiles found in \cite[Definition 4.4.1]{paper0-arxiv}.
\end{example}

\begin{definition}\label{definition:consistency-respecting}
    Let $(\enc, \app)$ be Self-Reference System and $\D \subseteq \E \setminus \{\bot\}$ be a subset of valid elements with respect to a meet semi-lattice $\wedge$. We say that an expression $\T \in \E$ is \term{consistency-respecting} if $\forall d, d' \in \D$:
    \begin{equation*}
        d \const d' \text{ if and only if } \T \ast d \const \T \ast d'
    \end{equation*}
    Equivalently, for a consistency respecting $\T$: $d \inconst d'$ if and only if $\T \ast d \inconst \T \ast d'$ for all pairs $d, d' \in \D$ as well.
\end{definition}

\parsplit

In Arithmetic Logic, we show that in an $\omega$-consistent theory, its provability predicate is consistency-respecting in $(\NGodelvar, \app)$ with respect to $\L_0 \setminus \{\bot\}$ and $\wedge$ (see Example \ref{example:godel-consistent-pairs}). Furthermore, if the theory is assumed to be complete, a contradiction follows. The mutual exclusivity of consistency and completeness is the essence of Gödel's Incompleteness Theorem. Then, in Social Choice Theory, we will show that the existence of contradictory preference cycles that follow from Arrow's Impossibility Theorem implies that no consistency-respecting expressions exist for \textit{any} Self-Reference System encoding with such a Social Welfare Function. Conversely, for certain types of dictators, certain consistency-respecting expressions \textit{must} exist.

\parsplit
To prove Gödel's Incompleteness Theorem in terms of a consistency respecting provability predicates $\Provable(x)$, we begin with a number of important properties of $\Provable(x)$.

\parsplit

\begin{restatable}{proposition}{provabilitypredicates}\label{proposition:provability-predicate}
    Given the Self-Reference System $(\NGodelvar, \app)$ of Example \ref{example:godel-1}, $\Provable(x)$ satisfies:
    \begin{enumerate}
        \item $\Provable(\DObar{\ \top\ }) = \top$\\\vspace{4pt}(i.e., the provability of all tautologies is a tautology).
    \end{enumerate}
    Furthermore, if the underlying Arithmetic Logic is:
    \begin{enumerate}
        \setcounter{enumi}{1}
        \item \textit{complete} then for every $D \in \L_0$: $\Provable(\DObar{D}) \vee \Provable(\DObar{\neg D}) = \top$\vspace{4pt}\\
              (i.e., in classical logic either $D$ or $\neg D$ is true and completeness means one of the two must be provable.)
              \begin{equation*}
                  \text{This further implies that: }\neg \Provable(\DObar{D}) \wedge \neg \Provable(\DObar{\neg D}) = \bot
              \end{equation*}
        \item $\omega$-consistent then $\Provable(x)$ is consistency respecting as per Example \ref{example:godel-consistent-pairs}, which further implies that $\forall D \in \L_0$:
              \begin{equation*}
                  \Provable(\Obar{D}) \leq \neg \Provable(\Obar{\neg D})
              \end{equation*}
    \end{enumerate}
\end{restatable}
\begin{proof}
    See Appendix \ref{appendix:result-seciton-proofs}.\\
\end{proof}

We also note the following useful lemma of classical logic written in terms of Lindenbaum Algebras.
\begin{restatable}{lemma}{lemmaclassical}\label{lemma:classical}
    For $A, B \in \L_0$: $A \wedge B = \bot \iff A \leq \neg B$.
\end{restatable}
\begin{proof}
    See Appendix \ref{appendix:result-seciton-proofs}.\\
\end{proof}

Finally, we prove Gödel's Incompleteness Theorem in a manner it immediately follows that if $(\NGodelvar, \app)$ has a consistency-respecting provability predicate, it is incomplete.

\parsplit

\begin{theorem}[Gödel's Incompleteness Theorem]\label{theorem:godel-main}
    An $\omega$-consistent theory of Arithmetic Logic cannot be complete.
\end{theorem}
\begin{proof}
    \allowdisplaybreaks
    By Theorem \ref{theorem:arithmetic-logic-fixed-points} there exists a fixed-point $\G = \neg \Provable(\DObar{\G})$ in $\L_0$. Then, by our theory being $\omega$-consistent, Proposition \ref{proposition:provability-predicate}-3 implies that $\Provable(\DObar{\G}) \leq \neg \Provable(\DObar{\neg \G})$. Combining the two results, and negating both sides, we have that:
    \begin{equation}\label{equation:squeeze-1}
        \neg \G \leq \neg \Provable(\DObar{\neg \G})
    \end{equation}

    \parsplit

    \noindent Now, assume to the contrary that the theory is complete, this implies that:
    {\small
    \begin{align*}
        \bot  \quad  = \quad & \neg \Provable(\DObar{\G}) \wedge \neg \Provable(\DObar{\neg \G}) &  & \text{Proposition \ref{proposition:provability-predicate}-2} \\
        \quad  = \quad       & \G \wedge \neg \Provable(\DObar{\neg \G})                         &  & \text{Definition of $\G$}
    \end{align*}
    }
    Then, by Lemma \ref{lemma:classical} this is equivalent to the inequality:
    \begin{equation}\label{equation:squeeze-2}
        \neg \Provable(\DObar{\neg \G}) \leq \neg \G
    \end{equation}
    Combining equations (\ref{equation:squeeze-1}) and (\ref{equation:squeeze-2}) we have that $\G = \neg \Provable(\DObar{\neg \G})$. However, re-examining the completeness property, we have that $\bot = \G \wedge \G = \G$, which produces the following contradiction:
    {\small
    \begin{align*}
        \top  \quad  = \quad & \Provable(\DObar{\ \top\ }) &  & \text{Proposition \ref{proposition:provability-predicate}-1} \\
        \quad  = \quad       & \Provable(\DObar{\neg \G})  &  & \text{$\G = \bot$}                                           \\
        \quad  \leq \quad    & \neg \Provable(\DObar{ \G}) &  & \text{Proposition \ref{proposition:provability-predicate}-3} \\
        \quad  = \quad    & \G = \bot                   &  & \text{Definition of $\G$}
    \end{align*}
    }
    i.e., $\top = \bot$. Hence, the theory cannot be complete.\\
\end{proof}

\begin{theorem}\label{theorem:godel-new}
    If a theory Arithmetic Logic has a consistency-respecting provability predicate with respect to the subset of valid elements $\L_0 \setminus \{\bot\}$ and $\wedge$ then it is incomplete.
\end{theorem}
\begin{proof}
    In the proof of Theorem \ref{theorem:godel-main}, the only aspect of $\omega$-consistency used the identity $\forall D \in \L_0:\ \Provable(\Obar{D}) \leq \neg \Provable(\Obar{\neg D})$, which follows from $\Provable(x)$ being consistency respecting (see Proposition \ref{proposition:provability-predicate}-3).\\
\end{proof}

To demonstrate impossibility in Social Choice Theory in the same terms, we proceed by showing Arrow's Impossibility Theorem is equivalent to the statement that Social Welfare Functions satisfying IIA, Unanimity and Non-Dictatorship necessarily produce a contradictory preference cycle. Then, we show that no Self-Reference System using that Social Welfare Function as its encoding can have a consistency-respecting expression due to its aggregating of profiles to contradictory preference cycles.

\parsplit
\begin{restatable}[Arrow's Impossibility Theorem]{theorem}{arrowfull}
    \label{theorem:arrow-full}
    If a Social Welfare Function $\Welfare: \PCN \rightarrow \PC$ satisfies Unanimity, IIA and Non-Dictatorship then there exist profiles $q, q' \in \PN$ such that:
    \begin{enumerate}
        \item $\Welfare(q) = \Welfare(q') = \c$
        \item $q \inconst q'$
    \end{enumerate}
\end{restatable}
\begin{proof}
    Identifying all preference cycles by a single $\c$ in \cite[Theorem 4.3.4]{paper0-arxiv}, we have that $\exists q, q' \in \PN$ such that $\Welfare(q) = \Welfare(q') = \c$ and $q \wedge q' = (\c, \dots, \c)$, which in turn implies that $q \inconst q'$.\\
\end{proof}

These conditions further imply the following impossibility result:

\parsplit
\begin{theorem}\label{theorem:arrow-main}
    If $\Welfare: \PCN \rightarrow \PC$ is a Social Welfare Function satisfying Unanimity and IIA and Non-Dictatorship, then no Self-Reference System $(\Welfare, \app)$ has a consistency-respecting expression with respect to $\PN \subseteq \PCN \setminus \{\c,\dots,\c\}$ and $\bwedge$.
\end{theorem}
\begin{proof}
    Assume to the contrary that $\Welfare$ does not have a dictator but that there exists a Self-Reference System $(\Welfare, \app)$ with a consistency-respecting $\T$. By Theorem \ref{theorem:arrow-full}: $\exists\ q, q' \in \PN$ such that $q \inconst q'$ and $\Welfare(q) = \Welfare(q') = \c$. By $\T$ being consistency-respecting, $q \inconst q' \implies \T \ast q \inconst \T \ast q'$. Then, because:
    \begin{equation*}
        \T \ast q = \app(\T, \Welfare(q)) = \app(\T, \c) = \app(\T, \Welfare(q')) = \T \ast q'
    \end{equation*}
    we can conclude that $\T \ast q$ must be inconsistent with itself, i.e., $\T \ast q \inconst \T \ast q$. However, $\T$ being consistency-respecting implies that $q$ is inconsistent with itself as well. However, if $q$ is inconsistent with itself, then for some coordinate $j$: $q_j \wedge q_j = \c$, which implies that $q_j = \c$, which contradicts that $q \in \PN$.\\
\end{proof}

\begin{proposition}\label{proposition:strong-dictator}
    If a Social Welfare Function $\Welfare: \PCN \rightarrow \PC$ has a strong dictator at $i$ (Definition \ref{definition:strong-dictator}) then $(\i, \dots, \i) \in \PN$ is consistency respecting in the Self-Reference System $(\Welfare, \wedge_i)$ with respect to $\bwedge$.
\end{proposition}
\begin{proof}
    $\T_i$ is consistency respecting in this case if and only if for $p, p' \in \PN$:
    \begin{equation}\label{equation:strong-dictator}
        p \const p' \iff (\i, \dots, \i) \ast p \const (\i, \dots, \i) \ast p'
    \end{equation}
    and $(\i, \dots, \i) \ast p = (\i, \dots, i \wedge \Welfare(p), \dots, \i) = (\i, \dots, \Welfare(p), \dots, \i)$. Thus, $\T_i$ is consistency respecting if $p \const p' \iff \Welfare(p) \const \Welfare(p')$. However, if $\Welfare$ has a strong dictator then $\Welfare(p) = p_i$, hence, Equation (\ref{equation:strong-dictator}) holds.\\
\end{proof}
The above result does not hold for ordinary dictators because while in that case $\Welfare(p) \const \Welfare(p') \implies p \const p'$, the converse does not hold because where a dictator is indifferent, the aggregate may introduce strict preferences. Hence, if $\Welfare(p) \const \Welfare(p')$ that does not necessarily mean that $p \inconst p'$.

In this section, we demonstrated that the existence or non-existence of consistency respecting expressions characterised the mutual exclusivity of $\omega$-consistency and completeness that underlies Gödel's Incompleteness Theorem (Theorem \ref{theorem:godel-new}), as well as some aspects of impossibility in Social Choice Theory (Theorem \ref{theorem:arrow-main} and Proposition \ref{proposition:strong-dictator}). However, a closer overlap between Arrow's Impossibility Theorem and Gödel's Incompleteness Theorem can be shown by investigating a different system of constraints abstracted out of our proof of Gödel's Incompleteness Theorem (Theorem \ref{theorem:godel-main}) that is the subject of the next section.

\subsection{\subsectiontextg}\label{subsection:argument}

In this section, we demonstrate overlaps between Arrow's Impossibility Theorem and Gödel's Incompleteness Theorem by abstracting constraints satisfied by a provability predicate in theory of Arithmetic Logic. For example, by Proposition \ref{proposition:provability-predicate}, if a Theory is complete then for any proposition $D \in \L_0$:
\begin{equation}\label{equation:motivation}
    \neg \Provable(\DObar{D}) \wedge \neg \Provable(\DObar{\neg D}) = \bot
\end{equation}
Moreover, Equation (\ref{equation:motivation}) can be formulated for any pair of expressions $(\T, d)$ in a Self-Reference System with a meet semi-lattice $\wedge$ on $\E$ with bottom element $\bot$ and complements $\neg$, i.e.:
\begin{equation}\label{equation:motivation-2}
    \neg \T \ast d \wedge \neg \T \ast \neg d = \bot
\end{equation}
where $\neg \T \ast d = \neg (\T \ast d)$. Because $\T$ does not necessarily represent a notion provability, we will say any pair of expressions $(\T, d)$ satisfying Equation (\ref{equation:motivation-2}) satisfies \term{quasi-completeness} to emphasise that this condition only represents completeness in its likeness to Equation (\ref{equation:motivation}).

\parsplit

A Self-Reference System is a \term{quasi-Gödelian system} with respect to an expression $\T$ that abstracts the role of a provability predicate, if there exists another expression $d$ such that the pair $(\T,d)$ jointly satisfies a number of abstract properties (e.g., quasi-completeness) in a manner related to Gödel's Incompleteness Theorem. Specifically, a quasi-Gödelian system abstracts the mutual exclusivity of consistency and completeness that arises from the existence of a Gödel sentence (abstractly, $d$) that is defined with respect to a provability predicate (abstractly, $\T$).

\parsplit

To apply this to Social Choice Theory, we fix the familiar Self-Reference System $(\Welfare, \Omega_i)$ (Example \ref{example:arrow-1}) and the $i^{th}$ indicator (i.e., an expression) $\T_i$ (Definition \ref{definition:indicator}) as abstracting the role of a provability predicate. Interestingly, we find that where Arrow's Impossibility Theorem implies that a non-dictatorial Welfare Function $\Welfare$ aggregates some profile (i.e., an expression) $q$ to a contradictory preference welfare cycle, we find that the pair $(\T_i,q)$ renders the Self-Reference System $(\Welfare, \Omega_i)$ quasi-Gödelian (Theorem \ref{theorem:condorcet-abstract}) with respect to $\T_i$. Moreover, if $\Welfare$ has a dictator at $i$ then $(\Welfare, \Omega_i)$ is never quasi-Gödelian with respect to $\T_i$ as for every profile $p$, $p$ is never abstractly a Gödel sentence with respect to $\T_i$ (Theorem \ref{theorem:dictator-abstract}) and the abstract versions of consistency and completeness are jointly satisfiable (Theorem \ref{theorem:dictator-abstract-2}).

\parsplit

We thus demonstrate that key properties underlying incompleteness in Arithmetic Logic characterise the structure of certain Social Welfare Functions that produce contradictory preference cycles (see Table \ref{table:overlap}).

\parsplit

Before proceeding, recall that an \term{orthocomplemented lattice} is a partially ordered set $(L, \leq)$ with all greatest lower bounds $\wedge$, least upper bounds $\vee$, top $\top$, bottom $\bot$, and a negation map that assigns to each $l \in L$ a $\neg l \in L$ such that $l \wedge \neg l = \bot$ and $l \vee \neg l = \top$\footnote{More specifically $\neg$ specifies an involution, i.e., additionally satisfying $\neg \neg l = l$ and $l \leq m \implies \neg m \leq \neg l$.}. In logic, Lindenbaum Algebras $\L_n$ are partially ordered by implication with respect to a theory (see Section \ref{subsection:background-algebraic-logic}), and orthocomplemented lattices with respect to conjunction, disjunction, true, false and negation, respectively. In Social Choice Theory, preference relations $\PCN$ can be ordered by strictness and orthocomplemented lattices with respect to $\wedge$, $\vee$, $\i$, $\c$ and $\neg$ as defined in Section \ref{subsection:encodings}, respectively. Unless otherwise specified, we use the symbols $\wedge$, $\vee$, $\top$, $\bot$ and $\neg$ to define an arbitrary orthocomplemented lattice; we can always recover its partial order, which must be definable from $\vee$ and $\wedge$ as $x \leq y \iff x \vee y = y \iff x \wedge y = x$~\cite{lattice-order-theory}.

\parsplit

A Quasi-Gödelian system with respect to an expression $\T$ is a Self-Reference System where another expression $d$ exists such that a number of constraints involving $\T$ and $d$ are satisfied. We first define the individual constraints that any pair of expressions $(\T, d)$ may or may not satisfy, in any combination.

\parsplit

\begin{definition}\label{definition:quasi-defs}
    Let $(\enc, \app)$ be a Self-Reference System whose expressions $\E$ and constants $\C$ are both orthocomplemented lattices, $\D \subseteq \E \setminus \bot$ a subset of valid elements (Definition \ref{definition:me}), and $d$ an expression in $\D$. We say that the pair $(\T, d)$ satisfies:
    \begin{enumerate}
        \item \term{Quasi-consistency} when $d \leq \neg \T \ast \neg d$
        \item That $d$ is a \term{Quasi-Gödel Sentence} with respect to $\T$ when $d \leq \neg \T \ast d$
        \item \term{Quasi-completeness} when $\neg \T \ast d \inconst \neg \T \ast \neg d$
    \end{enumerate}
    where, for $X, Y \in E$: $\neg X \ast Y$ is defined as $\neg (X \ast Y)$.
\end{definition}

\parsplit

\begin{definition}[Quasi-Gödelian Systems]\label{definition:quasi-godel}
    Let $(\enc, \app)$ be a Self-Reference System as per Definition \ref{definition:quasi-defs}, and $\T$ an expression. We say that $(\enc, \app)$ is a \term{Quasi-Gödelian System} with respect to $\T$ if there exists an expression $d$ that is a quasi-Gödel Sentence with respect to $\T$, but that quasi-consistency and quasi-completeness constraints are mutually exclusive (i.e., the two properties cannot be satisfied simultaneously).
\end{definition}

\parsplit

We now proceed to identify quasi-Gödelian systems in Arithmetic Logic and Social Choice Theory. See Table \ref{table:overlap} for a summary of these results.

\parsplit

\parsplit
\begin{theorem}\label{theorem:godel-abstract}
    If a theory of Arithmetic Logic is $\omega$-consistent then the Self-Reference System $(\NGodelvar, \app)$ with the subset of valid  elements $\L_0 \setminus \{\bot\}$ (Example \ref{example:godel-consistent-pairs}) is quasi-Gödelian with respect to $\Provable(x)$.
\end{theorem}
\begin{proof}
    Recall by Theorem \ref{theorem:arithmetic-logic-fixed-points} $\neg \Provable(x)$ has a fixed-point $\G \in \L_0$. Because $\G = \neg \Provable(\DObar{\G})$, $\G$ is clearly a quasi-Gödel sentence with respect to $\Provable(x)$ because $\neg\Provable(x) \ast \G = \neg\Provable(\DObar{\G})$.

    \parsplit

    To prove that $(\NGodelvar, \app)$ is a quasi-Gödelian with respect to $\Provable(x)$ it remains to show that the quasi-consistency and quasi-completeness of $(\Provable(x), \G)$ are mutually exclusive. Indeed, recalling that in our proof of Gödel's Incompleteness Theorem (Theorem \ref{theorem:godel-main}), we showed that if both $\omega$-consistency and completeness held, then $\G = \bot$ in $\L_0$, which in conjunction with $\Provable(\DObar{\ \top\ }) = \top$ led to a contradiction. In other words, we showed that $\omega$-consistency and completeness were mutually exclusive. Thus, because quasi-consistency and quasi-completeness follow from $\omega$-consistency and completeness, respectively by Proposition \ref{proposition:provability-predicate}, we have that they too are mutually exclusive, and thus $(\NGodelvar, \app)$ is a quasi-Gödelian system.\\
\end{proof}

\parsplit
\begin{theorem}\label{theorem:condorcet-abstract}
    Let $(\Welfare, \Omega_i)$ be the Self-Reference System with the subset of valid elements $\PN$ as per Example \ref{example:soc-me}, and let $\Welfare$ satisfy Unanimity and IIA. If $\Welfare$ satisfies Non-Dictatorship then $(\Welfare, \Omega_i)$ is a quasi-Gödelian system with respect to the $i^{th}$ indicator $\T_i \coloneqq (\c,\dots,\i,\dots,\c)$ (Definition \ref{definition:indicator}).
\end{theorem}
\begin{proof}
    \textbf{($q$ is a quasi-Gödel sentence with respect to $\T_i$):} by Theorem \ref{theorem:arrow-full} there is a $q \in \PN$ such that $\Welfare(q) = \c$, we find that $(\T_i, q)$ satisfies:
    \begin{equation}\label{equation:condorcet-abstract}
        \neg \T_i \ast q = \neg (\c, \dots, i \wedge \Welfare(q), \dots, c) = \neg (\c, \dots, \c, \dots, \c) = (\i, \dots, \i, \dots, \i)
    \end{equation}
    And $\top \coloneqq (\i, \dots, \i, \dots, \i)$ is the top element of $\PN$ so that $q \leq \neg \T_i \ast q$ holds.

    \parsplit

    \noindent\textbf{(Mutually exclusivity of quasi-consistency and quasi-completeness):} $(\T_i, q)$ satisfies quasi-completeness if and only if $\neg \T_i \ast q \inconst \neg \T_i \ast \neg q$. However, because $\neg \T_i \ast q = \top$, that occurs if and only if one coordinate of $\neg \T_i \ast \neg q$ is $\c$, which occurs if and only if $\Welfare(\neg q) = \i$ so that $\neg \T_i \ast \neg q = (\i, \dots, \c, \dots, \i)$. However, that occurs if and only if $q \nleq \neg \T_i \ast \neg q$ (i.e., quasi-consistency does not hold) because $q_i \nleq \c$, always.\\
\end{proof}

\begin{theorem}\label{theorem:dictator-abstract}
    Let $(\Welfare, \Omega_i)$ be the Self-Reference System with the subset of valid elements $\PN$ as per Example \ref{example:soc-me}, and let $\Welfare$ satisfy Unanimity and IIA. If $\Welfare$ satisfies Dictatorship and Unrestricted Domain then $(\Welfare, \Omega_i)$ is not a quasi-Gödelian system with respect to the $i^{th}$ indicator $\T_i$.
\end{theorem}
\begin{proof}
    For any valid expression $p \in \PN$, we have that $p$ is quasi-Gödel sentence with respect to $\T_i$ if and only if:
    \begin{equation*}
        \neg \T_i \ast p = \neg (\c, \dots, i \wedge \Welfare(p), \dots, c) = (\i, \dots, \neg \Welfare(p), \dots, \i)
    \end{equation*}
    and so $p \leq \neg \T_i \ast p$ if and only if $p_i \leq \neg \Welfare(p_i)$. By the assumption of $\Welfare$ having a dictator at $i$, if $p_i$ has any strict preference $a \prec b$ then $\neg \Welfare(p)$ must have the strict preference $b \prec a$, which means $p \nleq \neg \Welfare(p_i)$. The only other option is that $p_i = \i$, which means that the only solution for $p_i \leq \neg \Welfare(p_i)$ requires that $\Welfare(p_i) = \c$, which contradicts unrestricted domain.\\
\end{proof}

Quasi-completeness and quasi-consistency do not necessarily hold together for all dictatorial Social Welfare Functions, but they do in the following important special case:

\parsplit

\begin{theorem}\label{theorem:dictator-abstract-2}
    Let $(\Welfare, \Omega_i)$ be the Self-Reference System with the subset of valid elements $\PN$ as per Example \ref{example:soc-me}, and let $\Welfare$ satisfy Unanimity, IIA and have a dictator at individual $i$. If for $p \in \PN$, $p_i \neq \i$ and $\Welfare(\neg p) = \neg \Welfare(p)$ then quasi-completeness and quasi-consistency all hold in $(\T_i, p)$ for the $i^{th}$ indicator $\T_i$.
\end{theorem}
\begin{proof}
    \textbf{(Quasi Consistency)} Because $\Welfare(p) = \Welfare(\neg \neg p) = \neg \Welfare(\neg p)$, we thus have $w(p) \leq \neg w(\neg p)$.

    \parsplit

    \noindent\textbf{(Quasi Completeness)} $\neg \T_i \ast p \inconst \neg \T_i \ast \neg p$ if and only if $\Welfare(p) \inconst \Welfare(\neg p)$, which holds by the following argument. Firstly, because $p_i \neq \i$ the assumption that $\Welfare$ has a dictator at $i$ means that $\Welfare(p)$ and $\Welfare(\neg p)$ must have opposing, strict preferences for some pair of alternatives.\\
\end{proof}

\noindent The condition that $\Welfare(\neg p) = \neg \Welfare(p)$ is very relevant. In particular, it occurs when \term{neutrality} holds (i.e., when $\Welfare$ uses the same aggregation process for evaluating $a$ vs $b$ as any $c$ vs $d$). Neutrality for instance holds if for $\top \coloneqq (\i,\dots,\i)$, we have that $\Welfare(\top) = \i$ (a property known as \term{Pareto indifference}) along with IIA and Unrestricted Domain (e.g., see \cite{paper0-arxiv,arrow-hierarchy}).

    {\small
        \begin{table}[h]
            \centering
            \begin{tabular}{|l|l|l|l|}
                \hline
                \textbf{Setting}                 &
                \textbf{Arithmetic Logic}        &
                \textbf{No Dictator}             &
                \textbf{Unrestricted}                                                \\
                                                 &
                                                 &
                                                 &
                \textbf{Domain +}                                                    \\
                                                 &
                                                 &
                                                 &
                \textbf{Dictator at $i$}                                             \\ \hline
                Expression Pair                  &
                $(Provable(x),\ \G)$             &
                $(\T_i,\ q)$                     &
                $(\T_i,\ p)$                                                         \\

                                                 &
                for $\G$ a Gödel Sentence        &
                for $\Welfare(q) = \c$           &
                for $p \in \PN$                                                      \\\hline

                Quasi-Gödel Sentence             & \tick  & \tick  & \cross          \\\hline
                Quasi-consistency $\wedge$       & \cross & \cross & \sometimes      \\
                Quasi-completeness               &        &        &                 \\\hline

                Proven by Theorem                &
                \ref{theorem:godel-abstract}     &
                \ref{theorem:condorcet-abstract} &
                \ref{theorem:dictator-abstract} \& \ref{theorem:dictator-abstract-2} \\ \hline
            \end{tabular}

            \caption{A table of when the conditions of Definition \ref{definition:quasi-defs} hold. A tick (\tick) means that the condition always holds, a cross (\cross) means that the condition never holds, ``\sometimes'' means the condition sometimes holds.}
            \label{table:overlap}
        \end{table}
    }

We have thus demonstrated a significant overlap between Gödel's Incompleteness Theorem and Arrow's Impossibility Theorem in terms of the joint satisfiability of constraints that abstract Gödel Sentences, consistency and completeness in Arithmetic Logic. The mutual exclusivity of the quasi-consistency and quasi-completeness constraints characterised both Gödel's Incompleteness Theorem (Theorem \ref{theorem:godel-abstract}) and Social Welfare Functions that produce contradictory preference cycles (Theorem \ref{theorem:condorcet-abstract}) as necessitated by Arrow's Impossibility Theorem (Theorem \ref{theorem:arrow-full}). Moreover, where the incompatibility of quasi-consistency and quasi-completeness arise due to the existence of quasi-Gödel Sentences; if a Social Welfare Function has a dictator, quasi-consistency and quasi-completeness become compatible due to the absence of quasi-Gödel sentences (Theorems \ref{theorem:dictator-abstract}-\ref{theorem:dictator-abstract-2}). Although, the overlap between Gödel's Incompleteness Theorem and Arrow's Impossibility Theorem is appropriately not total because the reasons quasi-consistency and quasi-completeness become incompatible in each result differs.

\section{Discussion and Conclusion}\label{section:discussion}

Gödel's (First) Incompleteness Theorem maintains an ever-growing relevance to Formal Logic and Computer Science, largely due to its correspondence to theorems about the non-existence of algorithms for solving particular problems i.e., the incomputability of those problems. Incomputability results are valuable due to their ability to inform practitioners whether they are attempting to solve problems that are equivalent to well-known unsolvable problems~\cite{hoare-incomputability}. For example, the computation of certain fluid flows~\cite{fluid-flows}, ray-tracing paths in computer graphics~\cite{ray-tracing}, and air travel planning optimisations~\cite{air-travel-planning} are all equivalent to solving the incomputable Halting Problem.

\parsplit

Impossibility results in Social Choice Theory such as Arrow's Impossibility Theorem are crucial to Economics because they reveal inherent limitations in the design of decision-making systems that aggregate individual preferences into collective choices. Their implications are also applicable to informing practitioners of trade-offs that must be considered, for example, when developing voting methods~\cite{intro-to-voting-theory}, land management policy~\cite{land-management}, and economic indicators~\cite{social-welfare-gdp}.

\parsplit

Yet, a formal relationship between the fields' most seminal results: Arrow's Impossibility Theorem and Gödel's Incompleteness Theorem is lacking. In this paper, we related these two results by instantiating them in a general mathematical object we introduced, called a Self-Reference System. The resulting theory of Self-Reference Systems captured important properties of systems able to be characterised by encoding of expressions, and application of encodings to expressions. This was even demonstrated for highly lossy encoding functions such as Social Welfare Functions.

\parsplit

The most significant overlap arose through the abstraction of properties applicable to provability predicates. Specifically, we restated constraints on provability predicates (i.e., equations concerning Lindenbaum Algebras) using the general language of Self-Reference Systems. For example, a constraint characterising completeness could be stated for expressions (i.e., profiles) in Social Choice Theory. The abstract constraint (called \term{quasi-completeness}) does not concern completeness with respect to provability per se, but captures an important algebraic property that underlies completeness. The same process yielded constraints such as \term{quasi-consistency} and the existence of \term{quasi-Gödel sentences} (Definition \ref{definition:quasi-defs}).

\parsplit

Both Gödel's Incompleteness Theorem and Non-Dictatorship in Arrow's Impossibility Theorem (necessitating contradictory preference cycles, see Theorem \ref{theorem:arrow-full}) were shown to yield Self-Reference Systems where quasi-Consistency and quasi-Completeness were necessarily mutually exclusive due to the existence of a quasi-Gödel Sentence (Theorems \ref{theorem:godel-abstract} \& \ref{theorem:condorcet-abstract}). Self-Reference Systems failing to jointly satisfy the constraints in that manner were called \term{quasi-Gödelian}. However, the reasons that quasi-consistency and quasi-completeness were mutually exclusive in the two settings differed (see Table \ref{table:overlap}). Conversely, if a Social Welfare Function was dictatorial then it was not quasi-Gödelian because no quasi-Gödel sentence could exist and moreover, quasi-consistency and quasi-completeness were mutually satisfiable (Theorems \ref{theorem:dictator-abstract} \& \ref{theorem:dictator-abstract-2}). Hence, the fact that quasi-Gödelian systems arose precisely due to the presence of Gödel Sentences in Arithmetic Logic and Condorcet's Paradox in Social Choice Theory reveals a striking overlap between incomputability and impossibility between Gödel and Arrow's Theorems.

\parsplit

Moreover, Gödel Sentences $\G$, and profiles $q$ that aggregate to a contradictory preference cycle both manifested as expressions in Self-Reference Systems that encode to a bottom element, i.e., $\DObar{\bot}$ (Example \ref{example:godel-consistent-pairs}) and the contradictory preference cycle $\c$ (Example \ref{example:soc-me}), respectively. In both cases, this occurred with respect to a formal paradox. In logic, the paradox arose out of the assumption a theory was both $\omega$-consistent and complete (Theorem \ref{theorem:godel-main}), and in Social Choice Theory, the paradox arose out of the assumption of a Non-Dictatorial Social Welfare Function in the Arrovian Framework (Theorem \ref{theorem:arrow-full}) --- in essence, Condorcet's Paradox.

\parsplit

Another significant overlap was the correspondence between the well-known Diagonalisation Lemma used in proofs of Gödel's Incompleteness Theorem, and the dictatorship condition of a Social Welfare Function both corresponding to the existence of a diagonaliser in a Self-Reference Systems (Theorems \ref{theorem:arithmetic-logic-fixed-points}-\ref{theorem:dictator-as-diagonaliser}). However, the role of fixed-points in the two settings differed. The fixed-point property that underlies the Diagonalisation Lemma is a key element of Gödel's Incompleteness Theorem due to its entailment of Gödel sentences. However, in the context of Arrow's Impossibility Theorem, it prevents incomputability that follows from contradictory preference cycles that arise from Non-Dictatorship.

\parsplit

Hence, studying the two theorems in terms of generalised incomputability results of Self-Reference Systems revealed significant overlaps between the mechanisms underlying impossibility in Social Choice and incomputability in logic; however, not to the extent the two theorems were equated. Moreover, the fact that Gödel sentences were related to profiles that aggregate to contradictory preference cycles further implicates the role of circular and self-referential paradoxes in incomputability.

\parsplit

Before concluding, we outline a number of promising topics of further research towards the theory of Self-Reference Systems and its applications.

\parsplit

The theory of Self-Reference Systems may also be developed by identifying additional application domains to those of this paper; we outline three avenues towards further application domains. The first avenue is to analyse other well-known, impossibility or fixed-point results in terms of Self-Reference Systems, as we did for Arrow's Impossibility Theorem. For example, Chichilnisky's Impossibility Theorem of topological Social Choice Theory~\cite{chichilnisky-heal}, or the computability (i.e., existence) of Nash Equilibria (see~\cite{binmore} for a discussion of the relation of the concept to Diagonalisation). The second avenue is to investigate the overlap of the theory of Self-Reference System with existing general theories of Diagonalisation and Fixed-Point arguments. An example of such a general theory is Lawvere's Fixed-Point theorem~\cite{Lawvere}, which already has extensive applications~\cite{yanofsky}. However, Lawvere's Fixed-Point Theorem concerns functions with signature $\E \times \E \rightarrow \C$ rather than the signature $\E \times \C \rightarrow \E$, which is used for application functions in Self-Reference Systems. Thus, further generalisations of the theory of Self-Reference Systems may be required. The third avenue is to investigate the overlap of the theory of Self-Reference System with general theories of self-reference and self-reproduction. For example, Moss' Equational Logic of Self-Expression~\cite{moss}, Kauffman's Paired Categories~\cite{categorical-pairs} and Gonda et al.'s Simulators in Target Context Categories~\cite{gonda}.

\parsplit

To conclude, by introducing a theory of Self-Reference Systems, we were able to establish formal overlaps and differences between Arrow's Impossibility Theorem and Gödel's Incompleteness Theorem. This was achieved through abstractions of diagonalisation and fixed-point arguments, and of the joint satisfiability of key constraints regarding provability predicates in logic. Thus, broadening the scope of incomputability to include problems of Social-Decision Making in search of a more general foundation of computability may facilitate the cross-pollination of methods from fields with incomputability results.


\renewcommand*{\bibfont}{\scriptsize}
\printbibliography

\appendix

\section*{Appendix}
\renewcommand{\thesubsection}{\Alph{subsection}}

\subsection{Gödel Numbering on Lindenbaum Algebras}\label{appendix:godel-numbering}

Recalling that given a Gödel Numbering $G$ and formula $f \in \F_1$ that $\NGodel{f}$ is the `Gödel number' of the equivalence class of all logically equivalent formulae to $f$. By Definition \ref{definition:enc} we define $\NGodel{f}$ to be $G(f')$ for the shortest formula $f'$ logically equivalent to $f$. There are two important reasons for defining $\NGodel{f}$ this way. Firstly, it is useful to ensure that if two formulae $f, g \in \F_1$ are logically equivalent then for any predicate $B(x)$ so are $B(\DObar{f})$ and $B(\DObar{g})$. This is not necessarily the case when using the original Gödel numerals $\Obarvar$ rather than $\DObarvar$. For instance, if $f$ and $g$ are distinct but logically equivalent formulae then $G(f) \neq G(g)$. So, if $G(f) = n$, $G(g) = m$ and $B(x)$ is the predicate $\saymath{x = \underline{m}}$ then $B(\Obar{f}) = \saymath{\underline{n} = \underline{m}}$ is not logically equivalent to $B(\Obar{g}) = \saymath{\underline{m} = \underline{m}}$.

\parsplit

\noindent Secondly, we need to ensure $\NGodelvar$ is computable in order to reason about Gödel's Incompleteness Theory, e.g., to demonstrate that there is a Gödel Sentence. In other words, given $f \in \F_1$, the task of finding the shortest $f' \in \F_1$ that is logically equivalent to $f$ has to be achievable in finitely many steps.

\parsplit

\godelnumbering*
\begin{proof}
    1. By definition: $\NGodel{f} = \NGodel{g} = G(f')$ if and only if $f' \in \F_1$ is the shortest formula logically equivalent to both $f$ and $g$.
    \parsplit

    \noindent 2. If $f$ and $g$ are logically equivalent then $B(\DObar{f}) = B(\Obar{f'}) = B(\DObar{g})$.
    \parsplit

    \noindent 3. $\NGodelvar$ is computable because given a formula $f$, there are only finitely many formulae that are as short or shorter than $f$ which we \textit{need to check} for logical equivalence to $f$. Specifically, there are clearly finitely many as short or shorter formulae than $f$ if we ignore inconsequential changes of variable names (e.g., $x, y, z, \dots$ to $x_0, x_1, x_2, \dots$ or $a, b, c, \dots$).

    \parsplit

    To ignore these changes, one simply needs to fix an enumeration of variable names, say, $x_0, x_1, x_2, \dots$, and to find the shortest logically equivalent $f'$ to an $f$ on $n$ free variables, one need only consider the finitely many possible formulae as short as or shorter than $f$ that only use propositional variables $x_0, x_1, x_2, \dots, x_n$ (recalling that logically equivalent formulae must have the same number of variables as per Definition \ref{definition:lindenbaum-algebra-def}).\\
\end{proof}

\subsection{Proofs and Supplementary Results for Section \ref{section:results}}\label{appendix:result-seciton-proofs}

\begin{proposition}\label{proposition:bounds}
    Let $\P$ be the set of weak orders on a set $\A$ and $\leq$ be the strictness ordering on $\P$. Then, for weak orders $r, s \in \P$:
    \begin{enumerate}
        \item $r \leq s \implies r \subseteq s$ (i.e., as subsets of $\A \times \A$).
        \item The least upper bound $r \vee s$ is the transitive closure of $r \cup s$, $Trans(r \cup s)$.
        \item The greatest lower bound $r \wedge s$ if it exists is $r \cap s$.
    \end{enumerate}
\end{proposition}
\begin{proof}
    (1) We prove the contrapositive that $r \nsubseteq s$ implies $r \nleq s$. If $r \nsubseteq s$ then $\exists (a,b) \in r$ such that $(a,b) \notin s$. However, by the completeness axiom for weak orders, this implies $(b,a) \in s$, i.e., $b \prec a$ in $s$. However, if $(a,b)$ in $r$ then $b \nprec a$ in $r$. Hence, combining  $b \prec a$ in $s$ with $b \nprec a$ in $r$, we conclude $r \nleq s$.\\

    \noindent(2) We begin by noting that the relation $Trans(r \cup s)$ corresponds to a weak order because it is by definition transitive, and complete because $r$ is complete and $r \subseteq Trans(r \cup s)$. Then, it follows that $Trans(r \cup s)$ is an upper bound of $r$ and $s$ by (1). To show $Trans(r \cup s)$ is the least upper bound, it suffices to show for any other upper bound $t$ of $r$ and $s$, $Trans(r \cup s) \subseteq t$. Indeed, if a relation $t \in \P$ is a upper bound of $r$ and $s$, by (1), $r \subseteq t$ and $s \subseteq t$, which implies $r \cup s \subseteq t$. By definition of transitive closures, $Trans(r \cup s)$ includes all other transitive relations that include $r \cup s$. Hence, $Trans(r \cup s) \subseteq t$.\\

    \noindent(3) The intersection of two transitive relations is again transitive, so if $r \cap s$ is complete it corresponds to a weak order $v$ which is a lower bound of $r$ and $s$. By (1) to be the greatest lower bound, the relation must be the largest relation among lower bounds. Indeed, if we could remove an element $(a, b)$ from $r \cap s$ and have it still be a lower bound, then $b \prec a$ in $r$ and $s$ by completeness. But this means $(a, b)$ is not in either of $r$ or $s$, contradicting $(a,b) \in r \cap s$.\\
\end{proof}

\begin{proposition}\label{proposition:missing-bounds}
    $\PC$ (see Definition \ref{definition:extending-order}) has all greatest lower bounds.
\end{proposition}
\begin{proof}
    $\PC = \P \cup \{\c\}$ with bottom $\c$, we need only show that $r \wedge s$ exists for all incomparable elements $r$ and $s$ (with respect to $\leq$) in $\P$. Moreover, if $r$ and $s$ are incomparable then neither of them are $\c$, i.e., $r, s \in \P$. Hence, $r$ and $s$ must have opposing strict preferences (i.e., $a \prec b$ in $r$ and $b \prec a$ in $s$). By definition of $\c$ being a bottom element, $\c$ is the greatest lower bound of $r$ and $s$ if they have no other lower bound. Indeed, when $r$ and $s$ have opposing strict preferences, they do not have a lower bound in $\P$ and thus $\c$ is the only and hence greatest lower bound of $r$ and $s$.\\
\end{proof}

\dictatorequivalence*
\begin{proof}
    $(\implies)$ Assuming $\Welfare$ has a dictator at $i$, and considering an arbitrary $p \in \PN$, we prove $\Welfare(p) \join p_i = p_i$ as follows: Firstly, if $p_i = \i$, then $\Welfare(p) \join p_i = \Welfare(p) \join \i = \i = p_i$ is satisfied for all possible values of $\Welfare(p) \in \PC$. Otherwise, if $p_i \neq \i$ then Condition (1) implies $\Welfare(p) \neq \c$, which implies $\Welfare(p) \join p_i = \Welfare(p) \vee p_i$. But Condition (2) of Definition \ref{definition:dictator-generalised} implies $\Welfare(p) \leq p_i$, which implies $p_i = \Welfare(p) \vee p_i$. Combining the two facts, $\Welfare(p) \join p_i = \Welfare(p) \vee p_i = p_i$ as desired.\\

    \noindent $(\impliedby)$ Assuming $\Welfare(p) \join p_i = p_i$ always holds, we prove both conditions of Definition \ref{definition:dictator-generalised} hold as follows: For condition (1), if $\Welfare(p) = \c$ then $\forall r \in \PC$: $\Welfare(p) \join r = \i$. Hence, $\i = \Welfare(p) \join p_i = p_i$ as desired. For Condition (2), if $\Welfare(p) \neq \c$ then $\Welfare(p) \join p_i = \Welfare(p) \vee p_i$, and combined with our assumption that $\Welfare(p) \join p_i = p_i$, we have $\Welfare(p) \vee p_i = \Welfare(p) \join p_i = p_i$, and $\Welfare(p) \vee p_i = \Welfare(p)$ implies $\Welfare(p) \leq p_i$ as desired.\\
\end{proof}

\parsplit

\begin{lemma}
    \label{lemma:helper}
    Given an Embeddable Self-Reference System $(\enc, \emb, \comp)$, for $\DObarvar \coloneqq \emb \circ \enc$ we have:
    \begin{enumerate}
        \item $f \ast g = f \comp \DObar{g}$
        \item $(f \comp g) \ast h = f \comp (g \ast h)$
    \end{enumerate}
\end{lemma}
\begin{proof}
    This is given by calculations:
    \begin{enumerate}
        \item $f \ast g = \app(f, \enc(g)) = f \comp \emb(\enc(g)) = f \comp \DObar{g}$
        \item $(f \comp g) \ast h = (f \comp g) \comp \DObar{h} = f \comp (g \comp \DObar{h}) = f \comp (g \ast h)$
    \end{enumerate}
\end{proof}

\abstractdiagonalisationlemma*
\begin{proof}
    \allowdisplaybreaks
    For an arbitrarily $Q \in \E$, we define $q \coloneqq Q \comp f_\diag$, and $p \coloneqq q \ast q$, and find that $Q \ast p = p$ by the following calculation:
    \begin{align*}
        Q \ast p =\  & Q \comp \DObar{q \ast q}  &  & \text{Lemma \ref{lemma:helper}-1 and definition of $p$}       \\
        =\           & Q \comp (f_\diag \ast q)  &  & \text{$f_\diag$ internalising $\diag$ with respect to $\ast$} \\
        =\           & (Q \comp  f_\diag) \ast q &  & \text{Lemma \ref{lemma:helper}-2}                             \\
        =\           & q \ast q                  &  & \text{Definition of $q$}                                      \\
        =\           & p                         &  & \text{Definition of $p$}                                      \\
    \end{align*}
\end{proof}

\lemmaclassical*
\begin{proof}
    In $\L_0$, complements satisfy $\forall X \in \L_0$: $X \wedge \neg X = \bot$ and $X \vee \neg X = \top$. We have that $A = A \wedge \neg B$ by the following derivation:
    \begin{align*}
        A\  & =\quad  A \wedge \top                       &  & \text{$\top$ is the top of $\L_0$}            \\
            & =\quad  A \wedge (B \vee \neg B)            &  & \text{$B \vee \neg B = \top$}                 \\
            & =\quad  (A \wedge B) \vee (A \wedge \neg B) &  & \text{Distributivity of $\wedge$ over $\vee$} \\
            & =\quad  \bot \vee (A \wedge \neg B)         &  & \text{Assumption that $A \wedge B = \bot$}    \\
            & =\quad  A \wedge \neg B                     &  & \text{$\bot$ is the bottom of $\L_0$}
    \end{align*}
    and because $\wedge$ is a meet semi-lattice, we have that $A = A \wedge \neg B \iff A \leq \neg B$.\\
\end{proof}

\provabilitypredicates*
\begin{proof}
    We first note that the Hilbert–Bernays-Löb provability conditions~\cite[p.245]{peter-smith-tarski} apply to $\Provable(x)$. Below are two of these conditions that hold for any two formulae $X, Y \in \F_1$:
    \begin{enumerate}
        \item $\Theory \vdash X$ implies that $\Theory \vdash Provable(X)$.
        \item $\Theory \vdash \Provable(\Obar{X \rightarrow Y}) \rightarrow (\Provable(\Obar{X}) \rightarrow \Provable(\Obar{Y}))$.
    \end{enumerate}

    \noindent (1) $\Provable(\DObar{\ \top\ }) = \top$ because $\Theory \vdash \Provable(\DObar{\ \top\ }) \rightarrow \top$ holds by definition of $\top$, and $\Theory \vdash \top \rightarrow \Provable(\DObar{\ \top\ })$ holds by the first provability condition we listed.

    \parsplit

    \noindent (2) Completeness means that $\Theory \vdash D$ or $\Theory \vdash \neg D$, and by the first provability condition: $\Theory \vdash \Provable(\DObar{D})$ or $\Theory \vdash \neg \Provable(\DObar{\neg D})$. Thus, which ever is the case, we use disjunction introduction to get the desired identity.

    \parsplit

    \noindent (3) Recall that $\forall D \in \L_0$: $\Provable(x) \ast D = Provable(\DObar{D})$. This means that $\Provable(x)$ is consistency respecting when $\forall A, B \in \L_0 \setminus \{\bot\}$:
    \begin{equation*}
        A \wedge B = \bot \iff \Provable(\Obar{A}) \wedge \Provable(\Obar{B}) = \bot
    \end{equation*}
    We proceed to prove the above identity over the larger domain $\L_0$.

    \parsplit

    Firstly, we prove preservation, i.e., that $\forall A, B \in \L_0$
    \begin{equation*}
        A \wedge B = \bot \implies \Provable(\Obar{A}) \wedge \Provable(\Obar{B}) = \bot
    \end{equation*}
    Indeed, by Lemma \ref{lemma:classical}, we have that $A \leq \neg B$, i.e., $\Theory \vdash A \rightarrow \neg B$. Then by provability condition 1, we have that $\Theory \vdash \Provable(\Obar{A \rightarrow \neg B})$, which by provability condition 2, yields: $\Theory \vdash \Provable(\Obar{A}) \rightarrow \Provable(\Obar{\neg B})$.

    \parsplit

    \noindent Now, if $\Theory$ is $\omega$-consistent then $\Theory \vdash \Provable(\Obar{\neg B}) \rightarrow \neg \Provable(\Obar{B})$; thus we have $\Theory \vdash \Provable(\Obar{A}) \rightarrow \neg \Provable(\Obar{B})$.

    \parsplit

    Returning to $\L_0$ we have that $\Provable(\DObar{A}) \leq \neg \Provable(\DObar{B})$, which by Lemma \ref{lemma:classical} gives us the desired $\Provable(\Obar{A}) \wedge \Provable(\Obar{B}) = \bot$.

    \parsplit

    Secondly, we must show reflection, i.e., that:
    \begin{equation*}
        \Provable(\DObar{A}) \wedge \Provable(\DObar{B}) = \bot \implies A \wedge B = \bot
    \end{equation*}
    By provability condition 1 we have that $A \leq \Provable(\DObar{A})$ and $B \leq \Provable(\DObar{B})$, hence, $A \wedge B \leq \Provable(\DObar{A}) \wedge \Provable(\DObar{B}) = \bot$, which combined gives us $A \wedge B = \bot$ as $\bot$ is the bottom element of $\L_0$.

    \parsplit

    Finally, we attain $\Provable(\DObar{D}) \leq \neg \Provable(\DObar{\neg D})$ by invoking the above identities with $A = D$ and $B = \neg D$.\\
\end{proof}

\subsection{Outline of the Main Results}\label{appendix:guide}

\begin{longtable}{|p{0.2\linewidth}|p{0.75\linewidth}|}
    \hline
    \textbf{Section} \ref{subsection:encodings}     & \textbf{\subsectiontexta}                                                                                                                                                                                                                                                                                                                                                                                                                                                                                                                                                                                                \\\hline
    General Theory:                                 & An \term{encoding} is simply a function $\enc: \E \rightarrow \C$, from a set of \term{expressions} to a set of \term{constants}.                                                                                                                                                                                                                                                                                                                                                                                                                                                                                        \\\hline
    Arithmetic Logic:                               & Expressions are the Lindenbaum Algebra $\L_1$ of a theory and Constants are numbers $\N$.
    Encoding $\NGodelvar: \L_1 \rightarrow \N$ is a variation of Gödel Numbering that is well-defined over Lindenbaum Algebras (Definition \ref{definition:enc} \& Proposition \ref{proposition:well-defined}).                                                                                                                                                                                                                                                                                                                                                                                                                                                                \\\hline
    Social Choice Theory:                           & Constants are $\PC \coloneqq \P \cup \{\c\}$ for weak orders $\P$ and an object $\c$ representing contradictory preference cycles. An encoding is a Social Welfare Function whose domain and codomain are extended over $\c$ (Definitions \ref{definition:complete-condorcet-paradox} \& \ref{definition:extending-order}).                                                                                                                                                                                                                                                                                              \\\hline

    \textbf{Section} \ref{subsection:first-def}     & \textbf{\subsectiontextb}                                                                                                                                                                                                                                                                                                                                                                                                                                                                                                                                                                                                \\\hline
    General Theory:                                 & A \term{Self-Reference System} is a combination of an \term{encoding} and an \term{application} function
    $\app: \E \times \C \rightarrow \E$ (Definition \ref{definition:self-reference-system}).
    We define $e \ast f \coloneqq \app(e, \enc(f))$. Self-Reference arises out of expressions of the form: $e \ast e$.                                                                                                                                                                                                                                                                                                                                                                                                                                                                                                                                                         \\\hline
    Arithmetic Logic:                               & We take application to be variable substitution, i.e., $\app(B(x), n) = B(\underline{n})$ for the $n^{th}$ numeral $\underline{n}$.
    Self-reference arises out of propositions $B(x) \ast B(x) = B(\DObar{B(x)})$.                                                                                                                                                                                                                                                                                                                                                                                                                                                                                                                                                                                              \\\hline
    Social Choice Theory:                           & Application is typically defined as coordinate-wise usage of $\wedge$ and $\join$.
    For example, for a fixed individual $i$: $p \ast p$ is defined by replacing $p_i$ with $p_i \join \Welfare(p)$.
    We discuss expressions $p \ast p$ in terms of self-reference (Example \ref{example:arrow-1}).                                                                                                                                                                                                                                                                                                                                                                                                                                                                                                                                                                              \\\hline
    \textbf{Section} \ref{subsection:fp-property}   & \textbf{\subsectiontextc}                                                                                                                                                                                                                                                                                                                                                                                                                                                                                                                                                                                                \\\hline
    General Theory:                                 & The \term{Fixed-Point property} is satisfied for $e \in \E$ by $f \in \E$ if $e \ast f = f$.                                                                                                                                                                                                                                                                                                                                                                                                                                                                                                                             \\\hline
    Arithmetic Logic:                               & If the fixed-point property is satisfied for Expressions in $\L_1$ by expressions in $\L_0$ the Diagonalisation Lemma holds (Proposition \ref{proposition:diagonalisation-lemma-instance}).                                                                                                                                                                                                                                                                                                                                                                                                                              \\\hline
    Social Choice Theory:                           & A Social Welfare Function has a dictator if and only if in the Self-Reference System
    of Example \ref{example:arrow-1}, every profile $p \in \PN$ satisfies $p \ast p = p$ (Proposition \ref{proposition:dictator-as-fixed-point}).                                                                                                                                                                                                                                                                                                                                                                                                                                                                                                                              \\

    \hline
    \textbf{Section} \ref{subsection:embeddable}    & \textbf{\subsectiontextd}                                                                                                                                                                                                                                                                                                                                                                                                                                                                                                                                                                                                \\\hline
    General Theory:                                 & A Self-Reference System is \term{embeddable} if application can be expressed in terms of an \term{embedding} function $\emb: \C \rightarrow \E$ and a \term{composition} function $\comp: \E \times \E \rightarrow \C$ wherein $\app(e, f) = e \comp \DObar{f}$ for $\DObar{f} \coloneqq \emb(\enc(f))$.                                                                                                                                                                                                                                                                                                                 \\\hline
    Arithmetic Logic:                               & Embedding turns numbers into numerals and composition is proposition / predicate substitution (Example \ref{example:arrow-embedding}).                                                                                                                                                                                                                                                                                                                                                                                                                                                                                   \\\hline
    Social Choice Theory:                           & All the Self-Reference Systems of Social Choice Theory considered in this paper are embeddable by coordinate-wise uses of $\wedge$ and $\join$ (Example \ref{example:godel-embedding}).                                                                                                                                                                                                                                                                                                                                                                                                                                  \\\hline
    \textbf{Section} \ref{subsection:abstract-diag} & \textbf{\subsectiontexte}                                                                                                                                                                                                                                                                                                                                                                                                                                                                                                                                                                                                \\\hline
    General Theory:                                 & This theorem states the fixed-point property is satisfied for all expressions $e$ if a Self-Reference System has a \term{diagonaliser}, which is an expression $f_\diag$ such that $f_\diag \ast e = \DObar{e \ast e}$ (Theorem \ref{theorem:diag-abstract}).                                                                                                                                                                                                                                                                                                                                                            \\\hline
    Arithmetic Logic:                               & This holds in Arithmetic Logic with respect to a well-known diagonaliser predcate, and implies the usual Diagonalisation Lemma (Theorem \ref{theorem:arithmetic-logic-fixed-points}).                                                                                                                                                                                                                                                                                                                                                                                                                                    \\\hline
    Social Choice Theory:                           & This also holds in Social Choice Theory with respect to a diagonaliser called the $i^{th}$ indicator $\T_i$ (Definition \ref{definition:indicator}) and can characterise when a Social Welfare Functions has a dictator (Theorem \ref{theorem:dictator-as-diagonaliser}).                                                                                                                                                                                                                                                                                                                                                \\\hline
    \textbf{Section} \ref{subsection:const-resp}    & \textbf{\subsectiontextf}                                                                                                                                                                                                                                                                                                                                                                                                                                                                                                                                                                                                \\\hline
    General Theory:                                 & Given a meet semi-lattice $(\E, \wedge)$ of expressions with bottom $\bot$, a \term{Valid Subset} of $\E$ is a subset $\D \subseteq \E \setminus \{\bot\}$. $x, y \in \E$ are \term{consistent} when $x \wedge y \in \D$ (Definition \ref{definition:me}). An expression $\T \in \E$ is \term{consistency-respecting} if consistency between $d$ and $e$ are consistent if and only if $\T \ast d$ and $\T \ast e$ (Definition \ref{definition:consistency-respecting}).                                                                                                                                                 \\\hline
    Arithmetic Logic:                               & $\wedge$ and $\bot$ are their logical counterparts. In an $\omega$-consistent theory, its provability is consistency-respecting (Proposition \ref{proposition:provability-predicate}). This allows us to reformulate Gödel's Incompleteness Theorem (Theorem \ref{theorem:godel-new}).                                                                                                                                                                                                                                                                                                                                   \\\hline
    Social Choice Theory:                           & Non-dictatorship leading to contradictory preference cycles (Theorem \ref{theorem:arrow-full}) renders consistency-respecting expressions on $\PN$ impossible (Theorem \ref{theorem:arrow-main}), but possible given dictatorship (Proposition \ref{proposition:strong-dictator})                                                                                                                                                                                                                                                                                                                                        \\\hline
    \textbf{Section} \ref{subsection:argument}      & \textbf{\subsectiontextg}                                                                                                                                                                                                                                                                                                                                                                                                                                                                                                                                                                                                \\\hline
    General Theory:                                 & Given a Self-Reference System with a valid subset $\D$, we abstract elements of our proof of Gödel's Incompleteness Theorem (Theorem \ref{theorem:godel-main}), namely a pair of expressions $(\T, d)$ can satisfy \term{quasi-consistency}, \term{quasi-completeness}, and $d$ being \term{quasi-Gödel sentence} with respect to $\T$. A Self-Reference System is a \term{quasi-Gödelian system} with respect to $\T$ if there is a quasi-Gödel sentence $d$ of $\T$ such that \term{quasi-consistency} and \term{quasi-completeness} are mutually exclusive. \\\hline
    Arithmetic Logic:                               & Following our prior proof of Gödel's Incompleteness Theorem we see that the corresponding Self-Reference System is quasi-Gödelian with respect to $\Provable(x)$. This is because for $\G$ a (quasi-)Gödel sentence with respect to $\Provable(x)$, quasi-consistency (which follows from $\omega$-consistency) and quasi-completeness are mutually exclusive (Theorem \ref{theorem:godel-abstract}).                                                                                                                                                                               \\\hline
    Social Choice Theory:                           & If a Social Welfare Function produces contradictory preference cycle then $q$ is a quasi-Gödel sentence of the $i^{th}$ indicator $\T_i$ but quasi-consistency and quasi-completeness are shown to be mutually exclusive (Theorem \ref{theorem:condorcet-abstract}). However, if there is a dictator, there are no quasi-Gödel sentences with respect to $\T_i$ (Theorem \ref{theorem:dictator-abstract}), and the other constraints mutually satisfiable (Theorem \ref{theorem:dictator-abstract-2}).                                                                                        \\\hline
\end{longtable}

\end{document}